\documentclass{amsart}
\usepackage{color}
\usepackage{enumerate}
\usepackage{amsthm,amssymb}
\usepackage{amsmath}
\usepackage{graphicx,epsfig}
\usepackage{hyperref}

\allowdisplaybreaks[1]

\def\R{{\mathbb{R}}}

\newtheorem{theo}{Theorem}[section]
\newtheorem{lemma}[theo]{Lemma}
\newtheorem{prop}[theo]{Proposition}
\newtheorem{cor}[theo]{Corollary}

\newtheorem{hyp}[theo]{Hypothesis}
\theoremstyle{definition}

\newtheorem{rem}[theo]{Remark}
\newtheorem{defi}[theo]{Definition}

\def\calH{{\mathcal{H}}}

\def\R{\mathbb R}

\def\N{\mathbb N}

\newcommand{\norm}[1]{\left\Vert#1\right\Vert}
\newcommand{\abs}[1]{\left\vert#1\right\vert}

\let\div\undefined
\DeclareMathOperator{\div}{\mathrm{div}}

\newcommand{\bpr}{\begin{proof}}  
\newcommand{\epr}{\end{proof}}

\begin{document}

\numberwithin{equation}{section}

\title[Bounds for the gradient of transition kernels]{Bounds for the gradient of the transition kernel for elliptic operators with unbounded diffusion, drift and potential term}
\author{Markus Kunze, Marianna Porfido and Abdelaziz Rhandi}
\address{M. Kunze: Fachbereich Mathematik und Statistik, Universit\"at Konstanz, 78457 Konstanz, Germany}
\address{M. Porfido, A. Rhandi: Dipartimento di Matematica, Università degli Studi di Salerno, Via Giovanni Paolo II, 132, 84084 Fisciano (Sa), Italy}
\thanks{The second and the third authors are members of the
Gruppo Nazionale per l'Analisi Matematica, la Probabilità e le loro Applicazioni (GNAMPA) of
the Istituto Nazionale di Alta Matematica (INdAM). This article is based upon work
from COST Action 18232 MAT-DYN-NET, supported by COST (European Cooperation in Science
and Technology), www.cost.eu}
\vskip 0.5cm


\begin{abstract}
We prove global Sobolev regularity and pointwise upper bounds for the gradient of transition
densities associated with second order differential operators in $\R^d$ with unbounded diffusion, drift and potential terms.
\end{abstract}

\maketitle

\section{Introduction}
In this article, we are concerned with elliptic operators of the form
\begin{equation}\label{eq.div}
A\varphi ={\rm div}(Q\nabla \varphi)+F\cdot\nabla\varphi-V\varphi , \quad \varphi \in C^2(\R^d),
\end{equation}
where the diffusion coefficients $Q$, the drift $F$ and the potential $V$ are typically unbounded functions. Moreover, we write $A_0$ for the operator $A+V$. Throughout, we make the following assumptions on $Q$, $F$ and $V$.

\begin{hyp}\label{h.1}
We have $Q=(q_{ij})_{i,j=1, \ldots, d} \in C_{\mathrm{loc}}^{1+\zeta}(\R^d; \R^{d\times d})$, $F=(F_j)_{j=1, \ldots, d} \in C_{\mathrm{loc}}^{1+\zeta}(\R^d;\R^d)$ and $0\leq V \in C_{\mathrm{loc}}^\zeta(\R^d)$ for some $\zeta \in (0,1)$. Moreover, 
\begin{enumerate}
\item The matrix $Q$ is symmetric and uniformly elliptic, i.e. there is $\eta>0$ such that
\[
\sum_{i,j=1}^dq_{ij}(x)\xi_i\xi_j\ge \eta |\xi|^2 \quad \hbox{for all }x,\,\xi\in \R^d;
\]
\item there is $0\le Z\in C^2(\R^d)$ and a constant $M\ge 0$ such that $\lim_{|x|\to \infty}Z(x)=\infty$, $AZ(x)\le M$ and $\eta\Delta Z(x)+ F\cdot\nabla Z(x)-V(x)Z(x)\leq M$ for all $x\in \R^d$.
\item there is $0\le Z_0\in C^2(\R^d)$ and a constant $M\ge 0$ such that $\lim_{|x|\to \infty}Z_0(x)=\infty$, $A_0Z_0(x)\le M$ and $\eta\Delta Z_0(x)+ F\cdot\nabla Z_0(x)\leq M$ for all $x\in \R^d$.
\end{enumerate}
\end{hyp}

We immediately see that $(c)$ implies $(b)$, but in the applications will be useful to distinguish between $Z$ and $Z_0$.

It is well known  (see \cite[Theorem 2.2.5]{Lorenzi} and \cite{MPW}) that, assuming Hypothesis \ref{h.1}, 
a suitable realization of the above operator $A$ generates a (typically not strongly continuous) semigroup $T=(T(t))_{t\geq 0}$ on the space $C_b(\R^d)$ that is given through an integral kernel $p$, i.e.\ 
\[
T(t)f(x)=\int_{\R^d}p(t,x,y)f(y)\,dy,\quad t>0,\,x\in \R^d,\,f\in C_b(\R^d),
\]
where the kernel $p$ is positive, $p(t,\cdot ,\cdot)$ and $p(t,x,\cdot)$ are measurable for any $t>0,\,x\in \R^d$, and for a.e. fixed $y\in \R^d,\,p(\cdot ,\cdot ,y)\in C_{\mathrm{loc}}^{1+\frac{\zeta}{2},2+\zeta}((0,\infty)\times \R^d)$.\smallskip

An important aspect in the study of elliptic operators is to have estimates for the kernel $p$ and, consequently, this question has received a lot of attention in the literature. We mention here \cite{ALR10, AMP, BRS06, BS20, KunzeLorenziRhandi1, KunzeLorenziRhandi2, LMPR11, MPR, S13}, where specific operators were considered. We refer also to the manuscript \cite{BKRS15} and the references therein. There is also a general approach to establish estimates for $p$ making use of so called \emph{Lyapunov functions}, initiated in \cite{MPR} and later refined in \cite{ALR10, LMPR11}. 
We point out that these results are assuming the diffusion coefficients $Q$ to be uniformly bounded. However, using an approximation procedure, these results were subsequently extended to unbounded diffusion coefficients, see \cite{KunzeLorenziRhandi1, KunzeLorenziRhandi2}.\smallskip

In this article, we are concerned to establish not only estimates for $p$ but also for $\nabla p$, the gradient of $p$.  An important tool to obtain such estimates is the square integrability of the logarithmic gradient of $p$. Such integrability property plays an important role to obtain regularity results for $p$, cf. \cite[Section 7.4]{BKRS15}. Moreover, as in \cite{MPR}, once estimates for $\nabla p$ are obtained, one can repeat the same procedure to get estimates for $D^2p$ and hence estimates for $\partial_tp$. This allows us to obtain the differentiability of the semigroup $T(\cdot)$. 

Estimates for the gradient of $p$ were obtained in \cite[Section 5]{MPR} in the case of bounded diffusion coefficients. As in \cite{KunzeLorenziRhandi1, KunzeLorenziRhandi2}, we use approximation to extend this to unbounded diffusion coefficients.
We point out that the constant in the estimate for $\nabla p$ obtained in \cite[Thm.\ 5.3]{MPR} depend on $\norm{Q}_\infty$ and thus this estimate cannot be used in an approximation result. Therefore, we establish in Theorem \ref{Thm: weighted-gradient estimate bounded case} an estimate for $\nabla p$ in the case of bounded diffusion coefficients where the constant in the estimate does not depend on $\norm{Q}_\infty$.

With this estimate at hand, we can then tackle the case of unbounded diffusion coefficients by approximating them with bounded ones (see Section \ref{s.approx}). In this way, we can prove our main result Theorem \ref{Thm: weighted-gradient estimate} which provides an estimate of
$\nabla p$ in the general case. We illustrate our results by applying them to the prototype operator
\begin{equation}\label{eq.prototype}
\mathrm{div}\big((1+|x|_*^m)\nabla u\big) - |x|^{p-1}x\cdot\nabla u - |x|^s,
\end{equation}
where $x\mapsto |x|_*$ is a $C^2$-function satisfying $|x|_*=|x|$ for $|x|\geq 1$.

\subsection*{Notation}
$B_r$ denotes the open ball of $\R^d$ of radius $r$ and center $0$. For $0\leq a<b$, we write $Q(a,b)$ for $(a,b)\times \R^d$.

If $u: J\times \R^d\to \R$, where $J\subset [0,\infty)$ is an
interval, we use the following notation:
\begin{align*}
\partial_t u =&\frac{\partial u}{\partial t}, \
D_iu=\frac{\partial u}{\partial x_i}, \
D_{ij}u=D_iD_ju
\\
\nabla u=&(D_1u, \dots, D_du), \ {\div}(F)=\sum_{i=1}^dD_iF_i \hbox{\ for }F:\R^d\to \R^d,
\end{align*}
and
$$
|\nabla u|^2=\sum_{j=1}^d |D_j u|^2 , \qquad
|D^2u|^2=\sum_{i,j=1}^d|D_{ij} u|^2.
$$
Let us come to notation for function spaces.  $C_b(\R^d)$ is the
space of bounded and continuous functions in $\R^d$.  $\mathcal{D}(\R^d)$ is the space
of test functions. $C^\alpha(\R^d)$ denotes the space of all
$\alpha$-H\"older continuous functions on $\R^d$. $C^{1,2}(Q(a,b))$ is the space of all functions $u$ such that $\partial_tu$, $D_iu$ and $D_{ij}u$ are continuous in $Q(a,b)$. 
We also introduce
the space
$$C_c^{1,2}(Q(a,b))=\{\phi \in C^{1,2}(\overline{Q(a,b)}):{\rm
supp}\phi \subset [a,b]\times B_R \hbox{\ for some }R>0\}.$$ Notice
that we are not requiring that $u\in C_c^{1,2}(Q(a,b))$ vanishes at
$t=a,\,t=b$.

 For $\Omega \subseteq \R^d,\,1\le k\le \infty,\,j\in \N,\,W^{j}_{k}(\Omega)$
denotes the classical Sobolev space of all $L^k$--functions having
weak derivatives in $L^k(\Omega)$ up to the order $j$.  Its usual
norm is denoted by $\|\cdot \|_{j,k}$ and by $\|\cdot \|_k$ when
$j=0$. When $k=2$ we set $H^j(\Omega):=W^{j}_{2}(\Omega)$ and $H_0^1(\Omega)$ denotes the closure of 
the set of test functions on $\Omega$ with respect to the norm of $H^1(\Omega)$.

For $0<\alpha \le 1$ we denote by
$C^{1+\alpha/2 ,2+\alpha}(Q(a,b))$ the space of all
functions $u$ such that $\partial_tu$, $D_iu$ and $D_{ij}u$ are 
$\alpha$-H\"older continuous in $Q(a,b)$ with respect to
the parabolic distance $d((t,x),(s,y)):=|x-y|+|t-s|^{\frac{1}{2}}$.

Local H\"older spaces are defined, as usual, requiring that the
H\"older condition holds in every compact subset.

We shall also use parabolic Sobolev spaces.  We denote by
$W_k^{1,2}(Q(a,b))$ the space of functions $u\in L^k(Q(a,b))$
having weak space derivatives $D_i^\alpha u\in L^k(Q(a,b))$ for
$|\alpha|\le 2$ and weak time derivative $\partial_t u\in
L^k(Q(a,b))$ equipped with the norm
$$
\|u\|_{W_k^{1,2}(Q(a,b))}:=\|u\|_{L^k(Q(a,b))}
+\|\partial_t u\|_{L^k(Q(a,b))}
+\sum_{1\le |\alpha |\le 2}\|D^\alpha u\|_{L^k(Q(a,b))}.
$$
Let $\calH^{k,1}(Q(a,b))$ denote a space of all functions $u\in W_k^{0,1}(Q(a,b))$ with $\partial_t u\in (W_{k'}^{0,1}(Q(a,b)))'$, the dual space of $W_{k'}^{0,1}(Q(a,b))$, endowed with the norm
\[
\norm{u}_{\calH^{k,1}(Q(a,b))} := \norm{\partial_t u}_{(W_{k'}^{0,1}(Q(a,b)))'}+ \norm{u}_{W_k^{0,1}(Q(a,b))},
\]
where $1/k+1/k'=1$.

\section{Results for bounded diffusion coefficients}

Throughout this section we assume that the coefficients $q_{ij}$ and their spatial derivatives $D_kq_{ij}$ are bounded on $\R^d$ for all $i,j,k=1,\dots, d$.

As in \cite{ALR10}, \cite{Spina08} and \cite{KunzeLorenziRhandi2}, we introduce time dependent Lyapunov functions for $L:=\partial_t+A$.

\begin{defi}
We say that a function $W:[0,T]\times \R^d\to [0,\infty)$ is a \emph{time dependent Lyapunov function for} $L$ if $W\in C^{1,2}((0,T)\times \R^d)\cap C([0,T]\times \R^d)$ such that $\lim_{|x|\to \infty}W(t,x)=\infty$ uniformly for $t$ in compact subsets of $(0,T],\,W\le Z$ and there is $h\in [0,T]\to [0,\infty)$ integrable near $0$ such that
\begin{equation}\label{eq1: definition time dependent Lyapunov functions}
LW(t,x)\le h(t)W(t,x)
\end{equation}
and 
\begin{equation}\label{eq2: definition time dependent Lyapunov functions}
\partial_t W(t,x)+\eta\Delta W(t,x)+ F(x)\cdot\nabla W(t,x)-V(x)W(t,x)\leq h(t)W(t,x)
\end{equation}
for all $(t,x)\in (0,T)\times \R^d$. 
To emphasize the dependence on $Z$ and $h$, we also say that $W$ is a \emph{time dependent Lyapunov function for $L$ with respect to $Z$ and $h$}.
\end{defi}

Moreover, given a time dependent Lyapunov function $W$ for $L$, we define the function
$$\xi_W (t,x):=\int_{\R^d}p(t,x,y)W(t,y)\,dy.$$
This is finite due to the following result which is true also for possibly unbounded diffusion coefficients.

\begin{prop}\label{Prop: Time dependent Lyapunov functions are integrable with respect to pdy}
If $W$ is a time dependent Lyapunov function for $L$ with respect to $h$, then for $\xi_W (t,x):=\int_{\R^d}p(t,x,y)W(t,y)\,dy$, we have
\[
\xi_W (t,x)\le e^{\int_0^th(s)\,ds}W(0,x),\quad \forall (t,x)\in [0,T]\times \R^d.\]
\end{prop}
\begin{proof}
The proof is similar to the one given in \cite[Proposition 12.1]{KunzeLorenziRhandi1}.
\end{proof}

Making use of time dependent Lyapunov functions, we start by establishing pointwise upper bounds for the kernel $p$.
The following result can be deduced as in \cite[Theorem 12.4]{KunzeLorenziRhandi1} and \cite[Theorem 4.2]{KunzeLorenziRhandi2}.

\begin{theo}\label{theo-estimate bounded case}
Fix $T>0,\,x\in \R^d$ and $0<a_0<a<b<b_0<T$. Let us consider two time dependent Lyapunov functions $1\le W_1,\,W_2$ with $W_1\le W_2$ and a weight function $1\le w\in C^{1,2}((0,T)\times\R^d)$ such that 
\begin{enumerate}
\item the functions $w^{-2}\partial_t w$ and  $w^{-2}\nabla w$ are bounded on $Q(a_0,b_0)$;
\item there exist $k>d+2$ and constants $c_1,\ldots ,c_6\geq 1$, possibly depending on the interval $(a_0,b_0)$, with
\begin{tabbing}
\= (i) $w\le c_1w^{\frac{k-2}{k}}W_1^{\frac{2}{k}}$,
\qquad \= (ii) $|Q\nabla w|\le c_2w^{\frac{k-1}{k}}W_1^{\frac{1}{k}}$,
\qquad \= (iii) $|{\rm div}(Q\nabla w)|\le c_3w^{\frac{k-2}{k}}W_1^{\frac{2}{k}}$,\\[0.5em]
\> (iv) $|\partial_t w|\leq c_4 w^\frac{k-2}{k} W_1^\frac{2}{k}$,
\>(v) $V^{\frac{1}{2}}\le c_5w^{-\frac{1}{k}}W_2^{\frac{1}{k}}$,
\> (vi) $|F|\leq c_6 w^{-\frac{1}{k}}W_2^{\frac{1}{k}}$,
\end{tabbing}
on $[a_0,b_0]\times \R^d$.
\end{enumerate}
Then there is a constant $C>0$ depending only on  $d,\,k$ and $\eta$ such that
\begin{align}\label{Estimate for wp in the bounded case under h.2}
w(t,y)p(t,x,y) \leq C\Bigg[&c_1^{\frac{k}{2}}\sup_{t\in (a_0,b_0)}\xi_{W_1}(t,x)+\left(c_2^k+\frac{c_1^{\frac{k}{2}}}{(b_0-b)^{\frac{k}{2}}}+c_3^{\frac{k}{2}} +c_4^\frac{k}{2}\right) \int_{a_0}^{b_0}\xi_{W_1} (t,x)\, dt\notag \\
& +\left(c_5^k+c_6^k + c_2^\frac{k}{2} c_6^\frac{k}{2} \,\right)\int_{a_0}^{b_0}\xi_{W_2} (t,x)\, dt\Bigg],
\end{align}
for all $(t,y)\in (a,b)\times \R^d$ and any fixed $x\in \R^d$.
\end{theo}

\begin{rem}\label{Rem:supp}
If one assumes $|Q\nabla w|\le c_2W_1^{\frac{1}{2k}},\,|QD^2w|\le c'_3W_1^{\frac{1}{k}}$ and  
$|\nabla Q|\le c_7w^{-\frac{1}{k}}W_1^{\frac{1}{2k}}$, for some positive constants $c_2,c'_3,c_7$, then, since $w\ge 1$, we have
\begin{eqnarray}\label{eq:important}
\abs{\mathrm{div} (Q\nabla w)}
&\leq & d\left( \abs{\nabla Q} \abs{\nabla w}+\abs{Q D^2 w}\right) \nonumber \\
&\leq & d\left( c_2c_7\eta^{-1}w^{\frac{-1}{k}}W_1^{\frac{1}{k}}+c'_3W_1^{\frac{1}{k}}\right)\nonumber \\
&\leq & d\left( c_2c_7\eta^{-1}+c'_3\right)w^{\frac{k-2}{k}}W_1^{\frac{1}{k}}.
\end{eqnarray}
So, the assumption $(iii)$ of the above theorem is satisfied with $c_3=d\left( c_2c_7\eta^{-1}+c'_3\right)$, since $1\le W_1$.
\end{rem}

For further purposes, we obtain from the above remark the following corollary.
\begin{cor}\label{Rmk: constants in h.2 under h.3}
Assume all the assumptions of Theorem \ref{theo-estimate bounded case} except $(ii)$ and $(iii)$. If $|Q\nabla w|\le c_2W_1^{\frac{1}{2k}},\,|QD^2w|\le c'_3W_1^{\frac{1}{k}}$ and  
$|\nabla Q|\le c_7w^{-\frac{1}{k}}W_1^{\frac{1}{2k}}$ hold for some positive constants $c_2,c'_3,c_7$, then
there is a constant $C>0$ depending only on  $d,\,k$ and $\eta$ such that
\begin{equation}\label{Estimate for wp in the bounded case under h.3}
w(t,y)p(t,x,y) \leq C\left(A_1\sup_{t\in (a_0,b_0)}\xi_{W_1}(t,x)+A_2\, \Xi_1(a_0,b_0) +A_3\,\Xi_2(a_0,b_0)\right),
\end{equation}
with
\begin{align}\label{def: constants A_i}
&A_1=c_1^{\frac{k}{2}},\notag\\
&A_2 =c_2^k+\frac{c_1^{\frac{k}{2}}}{(b_0-b)^{\frac{k}{2}}}+c_3^{\frac{k}{2}} +c_4^\frac{k}{2},\notag\\
&A_3=c_5^k+c_6^k + c_2^\frac{k}{2} c_6^\frac{k}{2},
\end{align}
where $c_3$ is as in Remark \ref{Rem:supp} and 
\[
\xi_{W_i} (t,x):=\int_{\R^d}p(t,x,y)W_i(t,y)\,dy,
\quad  
\Xi_i(a_0,b_0):= \int_{a_0}^{b_0}\xi_{W_i} (t,x)\, dt
\]
for $i=1,2$.
\end{cor}

We aim to establish estimates for the derivatives of the kernel $p$.
To this purpose we make the following assumptions.

\begin{hyp}\label{h.3}
Fix $T>0,\,x\in \R^d$ and $0<a_0<a<b<b_0<T$. Let us consider two time dependent Lyapunov functions $1\le W_1,\,W_2$ with $W_1\le W_2$ and a weight function $1\le w\in C^{1,3}((0,T)\times\R^d)$ with $\partial_t\nabla w\in C((0,T)\times\R^d)$ such that for some $\varepsilon\in (0,1)$ and $k>2(d+2)$ the following hold true:
\begin{enumerate}
\item $\displaystyle \int_{\R^d} \left( \frac{1}{w(t,y)}\right)^{1-\varepsilon}dy<\infty$ and $\displaystyle\int_{Q(a,b)} \left( \frac{1}{w(t,y)}\right)^{1-\varepsilon} dt\,dy<\infty$;
\smallskip
\item the functions $w^{-2}\nabla w$, $w^{-2}\partial_t w$, $w^{-2} D^2 w$, $w^{-3}\nabla w \cdot \nabla w$, $w^{-2}\partial_t \nabla w$, $w^{-3} \partial_t w \nabla w$, $(\nabla w)^{-k-1} D^2 w$ and $(\nabla w)^{-k-1} \partial_t \nabla w$ are bounded on $Q(a_0,b_0)$;
\smallskip
\item there exist  constants $c_1,\ldots, c_{11}\geq 1$, possibly depending on the interval $(a_0,b_0)$, such that
\begin{tabbing}
\= (i) $w\le c_1w^{\frac{k-2}{k}}W_1^{\frac{1}{k}}$,
\qquad \qquad \= (ii) $|Q\nabla w|\le c_2W_1^{\frac{1}{2k}}$,
\qquad \qquad \= (iii) $|QD^2 w|\le c_3W_1^{\frac{1}{k}}$,\\[0.5em]
\> (iv) $|\partial_t w|\leq c_4 w^\frac{k-2}{k} W_1^\frac{1}{2k}$,
\>(v) $V^{\frac{1}{2}}\le c_5w^{-\frac{1}{k}}W_2^{\frac{1}{2k}}$,
\> (vi) $|F|\leq c_6 w^{-\frac{1}{k}}W_2^{\frac{1}{2k}}$,\\[0.5em]
\> (vii) $\abs{\nabla Q}\leq c_7 w^{-\frac{1}{k}} W_1^{\frac{1}{2k}}$,
\> (viii) $\abs{\nabla F}\leq c_8 w^{-\frac{1}{k}} W_2^{\frac{1}{k}}$, 
\> (ix) $\abs{\nabla V} \leq c_9 w^{-\frac{2}{k}} W_2^{\frac{2}{k}}$,\\[0.5em]
\> (x) $\abs{D^3 w}\leq c_{10} W_1^{\frac{3}{2k}}$
,
\> (xi) $|\partial_t\nabla w|\leq c_{11} W_1^\frac{1}{k}$,
\end{tabbing}
on $[a_0,b_0]\times \R^d$.
\end{enumerate}
\end{hyp}

From now on, we fix $0<a_0<a<a_1<b_1<b<b_0<T$ with $T>0,\,b-b_1\geq a_1-a\geq a-a_0$ and $x\in\R^d$. Moreover, we consider $p$ as a function of $(t,y)\in (0,T)\times \R^d$.

We prove some preliminary results.

\begin{theo}\label{Thm: similar to Thm 5.1,Lem 5.1, Thm 5.2 MPR}
Assume Hypothesis \ref{h.3}. Then the following statements hold.
\begin{enumerate}
\item The functions $p\log p$ and $p\log^2 p$ are integrable in $Q(a,b)$ and in $\R^d$ for all fixed $t\in [a, b]$ and
\begin{align*}
\int_{Q(a,b)} \frac{\abs{\nabla p(t,x,y)}^2}{p(t,x,y)}\,dt\,dy 
\leq &\frac{1}{\eta^2}\int_{Q(a,b)} \left(|F(y)|^2+V^2(y)\right) p(t,x,y)\,dt\,dy\notag\\
& + \int_{Q(a,b)} p(t,x,y)\log^2 p(t,x,y) \,dt\,dy\notag\\
&- \frac{2}{\eta} \int_{\R^d} [p(t,x,y)\log p(t,x,y)]_{t=a}^{t=b}dy<\infty.
\end{align*}
In particular, $p^\frac{1}{2}$ belongs to $W_2^{0,1}(Q(a,b))$.
\item $\nabla p\in L^s(Q(a_1,b_1))$ for all $1\leq s\leq \infty$.
\item $p\in W_{k/2}^{1,2}(Q(a_1,b_1))$.
\end{enumerate}
\end{theo}

\begin{proof}
(a), (b) and (c) can be deduced adapting respectively \cite[Theorem 5.1]{MPR}, \cite[Lemma 5.1]{MPR} and \cite[Theorem 5.2]{MPR} to operators with potential term. We note that Hypothesis \ref{h.3}(a) is used to infer integrability of the functions $p\log p$ and $p\log^2 p$.
\end{proof}

It is possible to prove even more regularity on $\nabla p$, as the following result shows.

\begin{theo}\label{Thm: Dp belongs to H^{k/2,1}}
Assume Hypothesis \ref{h.3}.
Then
$\nabla p\in \calH^{\frac{k}{2},1}(Q(a_1,b_1))$.
\end{theo}

\begin{proof}
In view of Theorem \ref{Thm: similar to Thm 5.1,Lem 5.1, Thm 5.2 MPR}, we are left to show that
\[
\partial_t \nabla p(\cdot,x, \cdot) \in  (W_{(k/2)'}^{0,1}(Q(a_1,b_1)))'.
\]
Let $\vartheta\in C^\infty(\R)$ such that $\vartheta(t)=1$ for $t\in [a_1,b_1]$, $\vartheta(t)=0$ for $t\leq a$, $t\geq b$, $0\leq \vartheta\leq 1$.
We define, for fixed $x\in \R^d$,
\[
q(t,y):=\vartheta^{k/2}(t)p(t,x,y).
\]
Consider $\varphi\in C_c^{1,2}(Q(a,b))$. By \cite[Lemma 2.1]{MPR}, we have
\[
\int_{Q(a,b)} (\partial_t \varphi(t,y) + A\varphi(t,y)) p(t,x,y)\, dt\, dy =\int_{\R^d}\left(p(b,x,y)\varphi(b,y)-p(a,x,y)	\varphi(a,y)\right)dy.
\]
Substituting $\vartheta^\frac{k}{2}\varphi$ instead of $\varphi$ in the previous equation, we get
\[
\int_{Q(a,b)} \left( q\partial_t \varphi -\langle Q\nabla \varphi, \nabla q\rangle + \langle F, \nabla \varphi\rangle q -V\varphi q+p\varphi \partial_t \vartheta^\frac{k}{2} \right)\, dt\, dy =0.
\]
We replace again $\varphi$ by the difference quotients with respect to the variable $y$
\[
\tau_{-h}\varphi(t,y)=\frac{\varphi(t,y-he_j)-\varphi(t,y)}{|h|},
\]
for $(t,y)\in Q(a,b) $, $0\neq h\in\R$ and we obtain
\begin{align*}
& \int_{Q(a,b)}  q\partial_t (\tau_{-h}\varphi)\, dt\, dy 
-\int_{Q(a,b)} \langle Q\nabla (\tau_{-h}\varphi), \nabla q\rangle \, dt\, dy 
+ \int_{Q(a,b)} \langle F, \nabla (\tau_{-h}\varphi)\rangle q \, dt\, dy \\
&-\int_{Q(a,b)} Vq(\tau_{-h}\varphi)\,dt\,dy+\int_{Q(a,b)} p(\tau_{-h}\varphi )\partial_t \vartheta^\frac{k}{2} \, dt\, dy = I_1-I_2+I_3-I_4+I_5=0,
\end{align*}
where
\begin{align*}
& I_1=\int_{Q(a,b)}  q\partial_t (\tau_{-h}\varphi)\, dt\, dy, 
\quad  I_2=\int_{Q(a,b)} \langle Q\nabla (\tau_{-h}\varphi), \nabla q\rangle \, dt\, dy,\\
& I_3=\int_{Q(a,b)} \langle F, \nabla (\tau_{-h}\varphi)\rangle q \, dt\, dy,\quad I_4=\int_{Q(a,b)} Vq(\tau_{-h}\varphi)\,dt\,dy,
&& I_5=\int_{Q(a,b)} p(\tau_{-h}\varphi )\partial_t \vartheta^\frac{k}{2} \, dt\, dy.
\end{align*}
By a change of variables we have
\[
I_1
= \int_{Q(a,b)} (\tau_{h} q) \partial_t \varphi\, dt\, dy
\]
and
\[
I_2
= \frac{1}{|h|}\int_{Q(a,b)} \left(\langle Q(y+he_j)\nabla \varphi(t,y), \nabla q(t,y+he_j)\rangle
- \langle Q(y)\nabla \varphi(t,y), \nabla q(t,y)\rangle \right)\, dt\, dy.
\]
Summing and subtracting $|h|^{-1}\int_{Q(a,b)}\langle Q(y+he_j)\nabla \varphi(t,y), \nabla q(t,y)\rangle\, dt\, dy$ in the previous expression yields
\[
I_2
= \int_{Q(a,b)} \left(\langle Q(y+he_j)\nabla \varphi(t,y), \nabla \tau_h q(t,y)\rangle
+ \langle \tau_h Q(y)\nabla \varphi(t,y), \nabla q(t,y)\rangle \right)\, dt\, dy.
\]
Similarly, we find that
\begin{eqnarray*}
I_3 &=& \int_{Q(a,b)} \left(\tau_h q(t,y) \langle F(y+he_j), \nabla \varphi(t,y)\rangle + q(t,y)\langle \tau_h F(t,y), \nabla \varphi(t,y)\rangle \right)\, dt\, dy,\\
I_4 &=& \int_{Q(a,b)} \left(\tau_h V(y)q(t,y)+V(y+he_j)\tau_h q(t,y)\right)\varphi(t,y)\,dt\,dy
\end{eqnarray*}
and 
\[
I_5=\int_{Q(a,b)} (\tau_h p) \varphi \partial_t \vartheta^\frac{k}{2} \, dt\, dy.
\]
Since $q_{ij}\in C_b^1(\R^d)$, applying the Cauchy-Schwarz inequality and H\"older's inequality we deduce that
\[
|I_2|
\leq c \left(\norm{\nabla \tau_h q}_{L^{k/2}(Q(a,b))} +\norm{\nabla q}_{L^{k/2}(Q(a,b))}\right)\norm{\varphi}_{W_{(k/2)'}^{0,1}(Q(a,b))}.
\]
Moreover,
\begin{align*}
|I_3|
\leq &\left(\int_{Q(a,b)} \abs{\tau_h q(t,y)}^\frac{k}{2} \abs{F(y+he_j)}^\frac{k}{2}\, dt\, dy\right)^\frac{2}{k}\norm{\varphi}_{W_{(k/2)'}^{0,1}(Q(a,b))}\\
&+ \left(\int_{Q(a,b)}  q^\frac{k}{2} \abs{\tau_h F}^\frac{k}{2}\, dt\, dy\right)^\frac{2}{k}\norm{\varphi}_{W_{(k/2)'}^{0,1}(Q(a,b))}\\
\leq &\norm{\tau_h q}_\infty^\frac{k-2}{k}\left(\int_{Q(a,b)}\frac{\abs{\tau_h p(t,y)}^2}{p} \, dt\, dy\right)^\frac{1}{k}\left(\int_{Q(a,b)}\abs{F(y+he_j)}^k p\, dt\, dy\right)^\frac{1}{k}\norm{\varphi}_{W_{(k/2)'}^{0,1}(Q(a,b))}\\
&+ \norm{q}_{L^\infty(Q(a,b))}^\frac{k-2}{k}\left(\int_{Q(a,b)} \abs{\tau_h F}^\frac{k}{2}q\, dt\, dy\right)^\frac{2}{k}\norm{\varphi}_{W_{(k/2)'}^{0,1}(Q(a,b))}.
\end{align*}
Similarly, we have
\begin{align*}
|I_4| \leq &\norm{\tau_h q}_\infty^\frac{k-2}{k}\left(\int_{Q(a,b)}\abs{V(y+he_j)}^kp\,dt\,dy\right)^{\frac{1}{k}}
\left(\int_{Q(a,b)}\frac{\abs{\tau_h p(t,y)}^2}{p} \, dt\, dy\right)^\frac{1}{k}\norm{\varphi}_{W_{(k/2)'}^{0,1}(Q(a,b))}\\
&+ \norm{q}_{L^\infty(Q(a,b))}^\frac{k-2}{k}\left(\int_{Q(a,b)} \abs{\tau_h V}^\frac{k}{2}p\, dt\, dy\right)^\frac{2}{k}
\norm{\varphi}_{W_{(k/2)'}^{0,1}(Q(a,b))}.
\end{align*}
Finally,
\[
|I_5|
\leq c \norm{\tau_h p}_{L^{k/2}(Q(a,b))}\norm{\varphi}_{W_{(k/2)'}^{0,1}(Q(a,b))}.
\]
Hence,
\begin{align*}
\Bigg| \int_{Q(a,b)} & (\tau_{h} q) \partial_t \varphi\, dt\, dy \Bigg|
\leq c\Bigg[ \norm{\nabla \tau_h q}_{L^{k/2}(Q(a,b))} +\norm{\nabla q}_{L^{k/2}(Q(a,b))}\\
&+\norm{\tau_h q}_\infty^\frac{k-2}{k}\left\Vert\frac{\tau_h p}{\sqrt{p}}\right\Vert_{L^2(Q(a,b))}^{\frac{2}{k}}\left(\int_{Q(a,b)}\abs{F(y+he_j)}^k p(t,y)\, dt\, dy\right)^\frac{1}{k}\\
&+ \norm{\tau_h q}_\infty^\frac{k-2}{k}\left\Vert\frac{\tau_h p}{\sqrt{p}}\right\Vert_{L^2(Q(a,b))}^{\frac{2}{k}}\left(\int_{Q(a,b)}\abs{V(y+he_j)}^kp(t,y)\,dt\,dy\right)^{\frac{1}{k}}\\
&+\norm{q}_{L^\infty(Q(a,b))}^\frac{k-2}{k}\left\{\left(\int_{Q(a,b)}\abs{\tau_h F}^\frac{k}{2} q(t,y)\, dt\, dy\right)^\frac{2}{k}+\left(\int_{Q(a,b)}\abs{\tau_h V}^\frac{k}{2} q(t,y)\, dt\, dy\right)^\frac{2}{k}\right\}\\
&+\norm{\tau_h p}_{L^{k/2}(Q(a,b))}
\Bigg]\norm{\varphi}_{W_{(k/2)'}^{0,1}(Q(a,b))}.
\end{align*}

As $p\in W_{k/2}^{1,2}(Q(a,b))$ by Theorem \ref{Thm: similar to Thm 5.1,Lem 5.1, Thm 5.2 MPR}(c), it follows that $\nabla \tau_h q\to \nabla D_j q$ in $L^\frac{k}{2}(Q(a,b))$ which implies the boundedness of $\norm{\nabla \tau_h q}_{L^\frac{k}{2}(Q(a,b))}$. Similarly, we may infer the boundedness of $\left\Vert\frac{\tau_h p}{\sqrt{p}}\right\Vert_{L^2(Q(a,b))}$ from Theorem \ref{Thm: similar to Thm 5.1,Lem 5.1, Thm 5.2 MPR}(a). 
As $\nabla p\in L^\infty(Q(a,b))$ by Theorem \ref{Thm: similar to Thm 5.1,Lem 5.1, Thm 5.2 MPR}(b), the difference quotients $\tau_h q$ converge weak* in $L^\infty(Q(a,b))$ to $D_j q$, where also $\norm{\tau_h q}_\infty$ is bounded.
Boundedness of the integrals involving $F$ can easily be deduced from the fact that $F\in C_{\mathrm{loc}}^{1+\zeta},\,V\in C_{\mathrm{loc}}^\zeta$ and the mean value theorem.
All together, we see that for a certain constant $C$, we have
\[
\abs{\int_{Q(a,b)} (\tau_{h} q) \partial_t \varphi\, dt\, dy}
\leq C \norm{\varphi}_{W_{(k/2)'}^{0,1}(Q(a,b))},
\]
for all $\varphi \in C_c^{1,2}(Q(a,b))$. By density, this estimate extends to $\varphi\in W_{(k/2)'}^{0,1}(Q(a,b))$ and it follows that the elements $\tau_h q$ are uniformly bounded in $(W_{(k/2)'}^{0,1}(Q(a,b)))'$.
Thus, by reflexivity, we see that as $h\to 0$ we find cluster-points in $(W_{(k/2)'}^{0,1}(Q(a,b)))'$. But testing against functions in $C_c^\infty(Q(a,b))$, we find that the only possible cluster point is $D_j q$. 
This yields $\partial_t D_j p\in (W_{(k/2)'}^{0,1}(Q(a,b)))'$ and finishes the proof.
\end{proof}

The following result, which is a version of \cite[Thm.\ 3.7]{KunzeLorenziRhandi2}, is the key to prove the main theorem of this section.

\begin{theo}\label{Thm 3.7 KunzeLorenziRhandi}
Let $q_{ij}\in C_{\mathrm{loc}}^\varsigma (\R^d)\cap C_b(\R^d)$ be such that $q_{ij}=q_{ji}$ for $i,j=1,\dots,d$ and such that $\langle Q(x)\xi,\xi\rangle\geq \eta\abs{\xi}^2$ for a certain $\eta>0$ and any $x,\xi\in\R^d$.\\
Further, let $0\leq a_0<b_0\leq 1$, $k>d+2$ and let functions $f\in L^\frac{k}{2}(Q(a_0,b_0))$, $h=(h_i)\in L^k(Q(a_0,b_0),\R^d)$ and $u\in \mathcal{H}^{p,1}(Q(a_0,b_0))\cap L^\infty(a_0,b_0;L^2(\R^d))$, for some $p>d+2$, be given such that $u(a_0)=0$ and
\begin{equation}
\int_{Q(a_0,b_0)}[\langle Q\nabla u,\nabla\psi\rangle+\psi\partial_t u]dtdx
= \int_{Q(a_0,b_0)} f\psi dtdx
+\int_{Q(a_0,b_0)} \langle h,\nabla \psi\rangle dtdx,
\end{equation}
for all $\psi\in C_c^\infty(Q(a_0,b_0))$. Then, $u$ is bounded and there exists a constant $C>0$, depending only on $\eta,d$ and $k$ (but not depending on $\norm{Q}_\infty$) such that
\begin{equation*}
\norm{u}_\infty\leq C(\norm{u}_{2}+\norm{f}_{\frac{k}{2}}+\norm{h}_k).
\end{equation*}
\end{theo}
\begin{proof}
The difference of \cite[Thm.\ 3.7]{KunzeLorenziRhandi2} and the above theorem is that in \cite[Thm.\ 3.7]{KunzeLorenziRhandi2},
the norm $\norm{u}_2$ in the right-hand side is replaced with $\norm{u}_{\infty, 2}$. However, inspecting the proof of \cite[Theorem 3.7]{KunzeLorenziRhandi2}, we see that basically the same proof works. Indeed, in the proof it is initially assumed that
$\norm{u}_{\infty, 2}\leq 1$. This assumption is needed to prove that for $A_\ell(t):=\lbrace u(t)\geq \ell\rbrace$ and $A_\ell:=\lbrace u\geq \ell\rbrace$ we have $|A_\ell|\leq 1$. However, this is still true under the weaker assumption $\norm{u}_2\leq 1$:
\begin{equation*}
|A_\ell|
< \int_{A_\ell} \ell^2\,dt dx
\leq  \int_{A_\ell} |u(t,x)|^2 \,dt dx
\leq \norm{u}_{2}
\leq 1.
\end{equation*}
As this is the only place where the $\norm{\cdot}_{\infty, 2}$-norm appears, the rest of the proof carries over verbatim.
\end{proof}

With the help of Theorem \ref{Thm 3.7 KunzeLorenziRhandi}, we can now prove an upper bound for $\abs{w\nabla p}$ that does not depend on the $\norm{\cdot}_\infty$-bound of the diffusion coefficients.

\begin{theo} \label{Thm: weighted-gradient estimate bounded case}
Assume Hypothesis \ref{h.3}. 
Then there is a constant $C>0$ depending only on $d$, $k$ and $\eta$ (but not depending on $\norm{Q}_\infty$) such that
\begin{align}\label{weighted-gradient estimate bounded case}
 |w(t,y)&\nabla p(t,x,y)|\notag\\
& \leq C\Bigg\lbrace B_1 \,\Xi_1(a_0,b_0)^\frac{1}{2}\norm{w p}_{L^\infty(Q(a,b))}^\frac{1}{2}
+\left(B_2 \, \Xi_1(a_0,b_0)^\frac{2}{k}+B_3\,\Xi_2(a_0,b_0)^\frac{2}{k}\right)\norm{w p}_{L^\infty(Q(a,b))}^\frac{k-2}{k}\notag\\
&+ \Bigg[B_1\left(\sup_{t\in (a_0,b_0)}\xi_{W_1}(t,x)\right)^\frac{1}{k} 
+ B_4\, \Xi_1(a_0,b_0)^\frac{1}{k}+B_5\,\Xi_2(a_0,b_0)^\frac{1}{k}\Bigg]\norm{w p}_{L^\infty(Q(a,b))}^\frac{k-1}{k}\notag\\
&+ \left( B_6 \,\Xi_1(a_0,b_0)^\frac{1}{2} +B_7\,\Xi_2(a_0,b_0)^\frac{1}{2}\right)\left(\int_{Q(a,b)}\frac{\abs{\nabla p}^2}{p}\,dt\, dy\right)^\frac{1}{2}\Bigg\rbrace
\end{align}
for all $(t,y)\in (a_1,b_1)\times \R^d$ and fixed $x\in\R^d$, where $B_i,\,i=1,\ldots ,7$ are positive constants
depending only on $c_i,\,i=1,\ldots, 11,\,b,\,b_1$ and $k$.
\end{theo}

\begin{proof}

We first prove the theorem assuming that the weight function $w$, along with its first order partial derivatives and its second order partial derivatives of the form $D_{ij}w$ and $\partial_t D_i w$ are bounded.
We fix $a_0<a<a_2<a_1<b_1<b_2<b<b_0$.
  
We show that 
\begin{equation}\label{eq4: Thm: weighted-gradient estimate bounded case}
\nabla (wp)\in \mathcal{H}^{\frac{k}{2},1}(Q(a_2,b_2))\cap L^\infty(Q(a_2,b_2)).
\end{equation}
We apply Theorem \ref{Thm: similar to Thm 5.1,Lem 5.1, Thm 5.2 MPR}(b) and Theorem \ref{Thm: Dp belongs to H^{k/2,1}} to infer that that $\nabla p\in \mathcal{H}^{\frac{k}{2},1}(Q(a_2,b_2))\cap L^\infty(Q(a_2,b_2))$. 
Moreover, by \cite[Lemma 12.4]{KunzeLorenziRhandi1} and Theorem \ref{theo-estimate bounded case}, we have that $p\in \mathcal{H}^{\frac{k}{2},1}(Q(a_2,b_2))\cap L^\infty(Q(a_2,b_2)) $. Thus, we get \eqref{eq4: Thm: weighted-gradient estimate bounded case}.

Let $\vartheta\in C^\infty(\R)$ such that $\vartheta(t)=1$ for $t\in [a_1,b_1]$, $\vartheta(t)=0$ for $t\leq a_2$, $t\geq b_2$, $0\leq \vartheta\leq 1$ and $\abs{\vartheta'}\leq \frac{2}{b_2-b_1}$. 
We define
\[
q(t,y):=\vartheta^{k/2}(t)p(t,x,y)
\]
and we note that $\nabla (wq)\in \mathcal{H}^{\frac{k}{2},1}(Q(a_2,b_2))\cap L^\infty(Q(a_2,b_2))$.
Moreover, given $\varphi\in C_c^\infty(Q(a_2,b_2))$, we write
\[
\psi(t,y):=\vartheta^{k/2}(t) w(t,y) D_h\varphi(t,y),
\]
with $h=1,\dots, d$.
For each $h=1,\dots, d$ we apply \cite[Lemma 2.1]{MPR}, which remains valid for operators with potential term. Hence
\begin{align*}
\int_{Q(a_2,b_2)} (\partial_t \psi(t,y) + A\psi(t,y)) p(t,x,y)\, dt\, dy =0.
\end{align*}
Integrating by parts, we get
\begin{align*}
\int_{Q(a_2,b_2)} \big[ p\partial_t\psi - \langle Q\nabla \psi, \nabla p\rangle +\langle F,\nabla \psi \rangle p-V\psi p\big] dt\,dy=0.
\end{align*}
Replacing the expression of the functions $\psi$ and $q$, after some computations we derive that
\begin{align*}
\int_{Q(a_2,b_2)} \Bigg[&\frac{k}{2} \vartheta' \vartheta^\frac{k-2}{2} wp (D_h\varphi) + wq (\partial_t D_h\varphi) - \langle Q\nabla w, \nabla q\rangle (D_h\varphi)- \langle Q \nabla D_h\varphi, w\nabla q\rangle\\
&+ \langle F, q\nabla w\rangle (D_h\varphi)+ \langle F, \nabla D_h\varphi\rangle wq- Vwq(D_h\varphi)+q(\partial_t w)(D_h\varphi)\Bigg]dt\,dy=0.
\end{align*}
Integrating by parts again in order to remove the derivative $D_h$ in front of $\varphi$, we have that
\begin{align}\label{eq1: Thm: weighted-gradient estimate bounded case}
\int_{Q(a_2,b_2)} \Bigg[& -\frac{k}{2} \vartheta' \vartheta^\frac{k-2}{2} w(D_h p)\varphi -\frac{k}{2} \vartheta' \vartheta^\frac{k-2}{2} p(D_h w)\varphi + (\partial_t D_h(wq))\varphi+\langle (D_hQ)\nabla w, \nabla q\rangle \varphi \notag\\
&+\langle Q(D_h\nabla w), \nabla q\rangle \varphi+\langle Q\nabla w, D_h \nabla q\rangle \varphi+w\langle (D_h Q)\nabla q, \nabla \varphi\rangle + \langle QD_h(w\nabla q), \nabla \varphi\rangle \notag\\
&-q\langle F,  D_h\nabla w\rangle\varphi- q\langle D_h F, \nabla w\rangle \varphi-(D_h q)\langle F, \nabla w\rangle \varphi- w(D_h q)\langle F, \nabla \varphi\rangle  \notag\\
&- q(D_h w)\langle F, \nabla \varphi\rangle-wq\langle D_h F, \nabla \varphi\rangle + Vw(D_h q)\varphi+ Vq(D_h w)\varphi+(D_h V)wq\varphi\notag\\
&-(\partial_t D_h w)q\varphi-(\partial_t w)(D_h q)\varphi\Bigg] dt\, dy=0.
\end{align}
Since
\begin{align*}
\int_{Q(a_2,b_2)} \langle Q\nabla w, D_h \nabla q\rangle \varphi \,dt\,dy = -\int_{Q(a_2,b_2)} \Big[ (D_h q)\mathrm{div}(Q\nabla w)\varphi + (D_h q)\langle Q\nabla w, \nabla \varphi\rangle \Big] dt\,dy
\end{align*}
and 
\begin{align*}
\int_{Q(a_2,b_2)}\langle QD_h(w\nabla q), \nabla \varphi\rangle \,dt\,dy
&=\int_{Q(a_2,b_2)}\Big[\langle Q\nabla D_h(wq),\nabla \varphi \rangle 
- q\langle QD_h(\nabla w),\nabla \varphi \rangle\\
& \qquad - (D_h q) \langle Q\nabla w, \nabla \varphi\rangle\Big]\,dt\,dy,
\end{align*}
we can adjust the terms in \eqref{eq1: Thm: weighted-gradient estimate bounded case} to obtain that
\begin{equation*}
\int_{Q(a_2,b_2)}[\langle Q\nabla u,\nabla\varphi\rangle+\varphi\partial_t u]\,dtdy
= \int_{Q(a_2,b_2)} f\varphi \,dt dy
+\int_{Q(a_2,b_2)} \langle h,\nabla \varphi\rangle \,dtdy,
\end{equation*}
where
\begin{align*}
u=&D_h(wq),\\
f=&\frac{k}{2} \vartheta'\vartheta^\frac{k-2}{2} w(D_h p)+\frac{k}{2} \vartheta'\vartheta^\frac{k-2}{2} p(D_h w) - \langle (D_hQ)\nabla w, \nabla q\rangle- \langle Q(D_h\nabla w), \nabla q\rangle\\
& + (D_h q)\, \mathrm{div}(Q\nabla w)+ q\langle D_h \nabla w, F\rangle + q\langle \nabla w, D_h F\rangle + (D_h q)\langle \nabla w,F\rangle-V w(D_h q)\\
&-V q(D_h w)- wq (D_h V)+(\partial_t D_h w)q+(\partial_t w)(D_h q),\\
h=& 2(D_h q) Q \nabla w-w (D_h Q) (\nabla q)+ q Q D_h\nabla w+ wF (D_hq)+ qF (D_hw)+wq (D_h F).
\end{align*}

We now want to apply Theorem \ref{Thm 3.7 KunzeLorenziRhandi} to the function $u$ and infer that there exists a constant $C$, depending only on $d$, $\eta$ and $k$, but not on $\norm{Q}_\infty$, such that
\begin{align*}
\norm{D_h(wq)}_\infty
\leq C \Big[ &\norm{D_h(wq)}_2 
+ \frac{k}{b_2-b_1} \norm{\vartheta^\frac{k-2}{2} w(D_hp)}_\frac{k}{2}
+ \frac{k}{b_2-b_1} \norm{\vartheta^\frac{k-2}{2} p(D_hw)}_\frac{k}{2} \\
&+\norm{\langle (D_hQ)\nabla w, \nabla q\rangle}_\frac{k}{2}
+\norm{\langle Q(D_h\nabla w), \nabla q\rangle}_\frac{k}{2}+ \norm{(D_h q) \mathrm{div}(Q\nabla w)}_\frac{k}{2}
\\
&+\norm{q\langle D_h \nabla w, F\rangle}_\frac{k}{2}
+\norm{q\langle \nabla w, D_h F\rangle}_\frac{k}{2}
+\norm{(D_h q)\langle \nabla w,F\rangle}_\frac{k}{2}+\norm{V w(D_h q)}_\frac{k}{2}\\
& +\norm{V q(D_h w)}_\frac{k}{2} +\norm{wq (D_h V)}_\frac{k}{2}
+ \norm{(\partial_t D_h w)q}_\frac{k}{2}+ \norm{(\partial_t w)(D_h q)}_\frac{k}{2}\\
&+ \norm{(D_h q)Q\nabla w}_k+\norm{ w (D_h Q) (\nabla q)}_k
+\norm{qQ D_h\nabla w}_k+\norm{wF(D_h q)}_k\\
&+\norm{qF(D_h w)}_k+\norm{wq D_h F}_k\Big].
\end{align*}
Summing over $h=1, \dots, d$ and since $\norm{\nabla(wq)}_\infty\geq \norm{w\nabla q}_\infty-\norm{q\nabla w}_\infty $ yields
\begin{align}\label{eq2: Thm: weighted-gradient estimate bounded case}
\norm{w\nabla q}_\infty
\leq C \Bigg[ &\norm{w\nabla q}_2+\norm{q\nabla w}_2 
+ \frac{k}{b_2-b_1} \norm{\vartheta^\frac{k-2}{2} w\nabla p}_\frac{k}{2}
+ \frac{k}{b_2-b_1} \norm{\vartheta^\frac{k-2}{2} p\nabla w}_\frac{k}{2} \notag\\
&+\norm{\langle \nabla Q \nabla w, \nabla q\rangle}_\frac{k}{2}
+\norm{Q D^2 w \nabla q}_\frac{k}{2}
+ \norm{(\nabla q) \mathrm{div}(Q\nabla w)}_\frac{k}{2}\notag\\
&+\norm{q (D^2 w) F}_\frac{k}{2}
+\norm{q\langle \nabla w, \nabla F\rangle}_\frac{k}{2}
+\norm{(\nabla q)\langle \nabla w,F\rangle}_\frac{k}{2}
+\norm{V w\nabla q}_\frac{k}{2}+\norm{V q\nabla w}_\frac{k}{2}\notag\\
& +\norm{wq \nabla V}_\frac{k}{2}
+ \norm{(\partial_t \nabla w)q}_\frac{k}{2}
+ \norm{(\partial_t w)(\nabla q)}_\frac{k}{2}
+ \norm{\langle Q\nabla w,\nabla q\rangle}_k
+\norm{ w (\nabla Q) (\nabla q)}_k\notag \\
&+\norm{qQ D^2 w}_k+\norm{w\langle \nabla q, F\rangle}_k
+\norm{q\langle \nabla w, F\rangle}_k
+\norm{wq \nabla F}_k\Bigg] +\norm{q\nabla w}_\infty.
\end{align}

We set
\[
P:=\int_{Q(a_2,b_2)}\frac{\abs{\nabla p}^2}{p}\,dt\, dy 
\]
and, for a sake of simplicity, we write $\Xi_i$ instead of $\Xi_i(a_2,b_2)$ to refer to $\int_{a_2}^{b_2}\xi_{W_i} (t,x)\, dt$ for $i=1,2$. 
We observe that $\Xi_1,\,\Xi_2<\infty$ by Proposition \ref{Prop: Time dependent Lyapunov functions are integrable with respect to pdy}. Moreover, thanks to Theorem \ref{Thm: similar to Thm 5.1,Lem 5.1, Thm 5.2 MPR}(a), we know that $P<\infty$.
Finally, we estimate the terms in the right hand side of \eqref{eq2: Thm: weighted-gradient estimate bounded case}. We start with $\norm{w\nabla q }_2$. Using H\"older's inequality and Hypothesis \ref{h.3}(c) one obtains

\begin{align*}
\norm{w\nabla q }_2^2
= \int_{Q(a_2,b_2)} w^2\abs{\nabla q}^2 dt\, dy
&\leq \norm{w\nabla q}_\infty \int_{Q(a_2,b_2)} \frac{\abs{\nabla q}}{\sqrt{q}}\sqrt{q} w\,dt\, dy\\
&\leq \norm{w\nabla q}_\infty \left(\int_{Q(a_2,b_2)} \frac{\abs{\nabla q}^2}{q}\,dt\, dy\right)^\frac{1}{2} \left(\int_{Q(a_2,b_2)} w^2 q\,dt\, dy\right)^\frac{1}{2}\\
&\leq c_1^\frac{k}{2} \norm{w\nabla q}_\infty \left(\int_{Q(a_2,b_2)} \frac{\abs{\nabla p}^2}{p}\,dt\, dy\right)^\frac{1}{2}\left(\int_{Q(a_2,b_2)} \xi_{W_1}(t,x)\,dt\right)^\frac{1}{2}\\
&= c_1^\frac{k}{2}\norm{w\nabla q}_\infty  P^\frac{1}{2} \Xi_1^\frac{1}{2}.
\end{align*}

Hence, we have
\[
\norm{w\nabla q}_2\leq c_1^\frac{k}{4} P^\frac{1}{4} \Xi_1^\frac{1}{4} \norm{w\nabla q}_\infty^\frac{1}{2}.
\]

Similarly, we get
\[
\norm{\vartheta^\frac{k-2}{2} w\nabla p}_\frac{k}{2}\leq c_1 P^\frac{1}{k} \Xi_1^\frac{1}{k}\norm{w\nabla q}_\infty^ \frac{k-2}{k},
\]

\[
\norm{\langle \nabla Q \nabla w, \nabla q\rangle}_\frac{k}{2}\leq  \eta^{-1} c_2 c_7 P^\frac{1}{k} \Xi_1^\frac{1}{k}\norm{w\nabla q}_\infty^\frac{k-2}{k} ,
\]

\[
\norm{Q D^2 w \nabla q}_\frac{k}{2}\leq c_3 P^\frac{1}{k} \Xi_1^\frac{1}{k}\norm{w\nabla q}_\infty^\frac{k-2}{k} ,
\]

\[
\norm{(\nabla q) \mathrm{div}(Q\nabla w)}_\frac{k}{2}\leq d(\eta^{-1}c_2 c_7+c_3) P^\frac{1}{k} \Xi_1^\frac{1}{k}\norm{w\nabla q}_\infty^\frac{k-2}{k},
\]
where we have applied here \eqref{eq:important}. Moreover,
\[
\norm{(\nabla q)\langle \nabla w,F\rangle}_\frac{k}{2}\leq \eta^{-1} c_2 c_6 P^\frac{1}{k} \Xi_2^\frac{1}{k}\norm{w\nabla q}_\infty^\frac{k-2}{k},
\]

\[
\norm{V w\nabla q}_\frac{k}{2}\leq c_5^2 P^\frac{1}{k} \Xi_2^\frac{1}{k}\norm{w\nabla q}_\infty^\frac{k-2}{k},
\]

\[
\norm{(\partial_t w)(\nabla q)}_\frac{k}{2}\leq c_4 P^\frac{1}{k} \Xi_1^\frac{1}{k}\norm{w\nabla q}_\infty^\frac{k-2}{k},
\]

\[
\norm{\langle Q\nabla w,\nabla q\rangle}_k\leq c_2 P^\frac{1}{2k} \Xi_1^\frac{1}{2k}\norm{w\nabla q}_\infty^\frac{k-1}{k},
\]

\[
\norm{ w (\nabla Q) (\nabla q)}_k\leq c_7 P^\frac{1}{2k} \Xi_1^\frac{1}{2k}\norm{w\nabla q}_\infty^\frac{k-1}{k},
\]

\[
\norm{w\langle \nabla q, F\rangle}_k\leq c_6 P^\frac{1}{2k} \Xi_2^\frac{1}{2k}\norm{w\nabla q}_\infty^\frac{k-1}{k}.
\]

Moreover, we estimate $\norm{q\nabla w}_2^2$ as follows:

\[
\norm{q\nabla w}_2^2
= \int_{Q(a_2,b_2)} q^2 \abs{\nabla w}^2\, dt\, dy
\leq \eta^{-2} c_2^2 \norm{wq}_{\infty} \int_{Q(a_2,b_2)} W_1^\frac{1}{k} q\, dt\, dy
 \leq \eta^{-2} c_2^2 \norm{wq}_{\infty}\Xi_1.
\]

Thus, we have

\[
\norm{q\nabla w}_2\leq \eta^{-1} c_2 \Xi_1^\frac{1}{2}\norm{wq}_\infty^ \frac{1}{2}.
\]

In a similar way, we obtain 

\[
\norm{\vartheta^\frac{k-2}{2} p\nabla w}_\frac{k}{2}\leq \eta^{-1} c_2 \Xi_1^\frac{2}{k}\norm{wq}_\infty^ \frac{k-2}{k} ,
\]

\[
\norm{q (D^2 w) F}_\frac{k}{2} \leq \eta^{-1} c_3 c_6  \Xi_2^\frac{2}{k}\norm{w q}_\infty^\frac{k-2}{k},
\]

\[
\norm{q\langle \nabla w, \nabla F\rangle}_\frac{k}{2}\leq \eta^{-1} c_2 c_8  \Xi_2^\frac{2}{k}\norm{w q}_\infty^\frac{k-2}{k},
\]

\[
\norm{V q\nabla w}_\frac{k}{2}\leq \eta^{-1} c_2 c_5^2 \Xi_2^\frac{2}{k}\norm{w q}_\infty^\frac{k-2}{k},
\]

\[
\norm{wq \nabla V}_\frac{k}{2}\leq c_9 \Xi_2^\frac{2}{k}\norm{w q}_\infty^\frac{k-2}{k},
\]

\[
\norm{(\partial_t\nabla w)q}_\frac{k}{2}\leq c_{11}  \Xi_1^\frac{2}{k}\norm{w q}_\infty^\frac{k-2}{k},
\]

\[
\norm{qQ D^2 w}_k\leq c_3  \Xi_1^\frac{1}{k}\norm{w q}_\infty^\frac{k-1}{k},
\]

\[
\norm{q\langle \nabla w, F\rangle}_k\leq \eta^{-1} c_2 c_6  \Xi_2^\frac{1}{k}\norm{w q}_\infty^\frac{k-1}{k},
\]

\[
\norm{wq \nabla F}_k\leq c_8  \Xi_2^\frac{1}{k}\norm{w q}_\infty^\frac{k-1}{k}.
\]
Finally, we get
\[
\norm{q\nabla w}_\infty\leq \norm{q}_\infty^\frac{k-1}{k} \norm{q(1+\abs{\nabla w}^2)^\frac{k}{2}}_\infty^\frac{1}{k}.
\]
We now estimate $\norm{q(1+\abs{\nabla w}^2)^\frac{k}{2}}_\infty$ by applying Theorem \ref{theo-estimate bounded case} with $w$ replaced by $\tilde{w}=(1+\abs{\nabla w}^2)^\frac{k}{2}$. 
First, we check the assumptions using Hypothesis \ref{h.3}(c):

\[
\tilde{w}^\frac{2}{k}
= 1+\abs{\nabla w}^2 \leq 1+ \eta^{-2} c_2^2 W_1^\frac{1}{k}\leq (1+ \eta^{-2} c_2^2) W_1^\frac{2}{k},
\]

\[
|Q\nabla \tilde{w}|
= k (1+\abs{\nabla w}^2)^\frac{k-2}{2} |(Q D^2 w )\nabla w|
\leq k \tilde{w}^\frac{k-2}{k} |Q D^2 w| |\nabla w|
\leq kc_3 \tilde{w}^\frac{k-1}{k} W_1^\frac{1}{k},
\]

\begin{align*}
|\mathrm{div} (Q\nabla \tilde{w})|
\leq &d |\nabla(Q\nabla\tilde{w})|
\leq d |\nabla Q| |\nabla \tilde{w}| + d |Q D^2\tilde{w}|\\
\leq &d |\nabla Q| k \tilde{w}^\frac{k-2}{k} |D^2 w||\nabla w| 
+d \Big[(k-2) \tilde{w}^{-\frac{2}{k}} |\nabla \tilde{w}| |QD^2 w| |\nabla w|\\
&+ k \tilde{w}^\frac{k-2}{k} |D^3 w| |Q\nabla w| +k \tilde{w}^\frac{k-2}{k} |QD^2 w| |D^2 w|\Big]\\
\leq & k d [\eta^{-2} c_2 c_3 c_7+(k-1)\eta^{-1}c_3^2+c_2 c_{10}]\tilde{w}^\frac{k-2}{k} W_1^\frac{2}{k},
\end{align*}

\[
|\partial_t \tilde{w}|\leq k (1+\abs{\nabla w}^2)^\frac{k-2}{2} |\nabla w| |\partial_t \nabla w|
\leq k \eta^{-1} c_2 c_{11} \tilde{w}^\frac{k-2}{k} W_1^\frac{2}{k},
\]

\[
\tilde{w}^\frac{1}{k} V^\frac{1}{2}
\leq (1+|\nabla w|)V^\frac{1}{2}
\leq (c_5+\eta^{-1} c_2 c_5)W_2^\frac{1}{k},
\]

\[
\tilde{w}^\frac{1}{k} |F|
\leq (1+|\nabla w|)|F|
\leq (c_6+\eta^{-1}c_2c_6)W_2^\frac{1}{k}.
\]

Moreover, $\tilde{w}^{-2}\nabla \tilde{w}$ and $\tilde{w}^{-2}\partial_t \tilde{w}$ are bounded on $Q(a_0,b_0)$ as we assume in Hypothesis \ref{h.3} that the functions $(\nabla w)^{-k-1} D^2 w$ and $(\nabla w)^{-k-1} \partial_t \nabla w$ are bounded.
Hence, the assumptions of Theorem \ref{theo-estimate bounded case} holds true with $w$ replaced by $\tilde{w}$ and with the constants $c_1, \dots, c_6$ replaced, respectively, by $1+\eta^{-2}c_2^2$, $kc_3$, $k d [\eta^{-2} c_2 c_3 c_7+(k-1)\eta^{-1}c_3^2+c_2 c_{10}]$, $k \eta^{-1} c_2 c_{11}$, $c_5+\eta^{-1} c_2 c_5$ and $c_6+\eta^{-1}c_2c_6$.
Thus, we obtain that 
\begin{align*}
\norm{q(1+\abs{\nabla w}^2)^\frac{k}{2}}_\infty
\leq C\Bigg[&c_2^k \sup_{t\in (a_2,b_2)}\xi_{W_1}(t,x)
+\left(c_3^k+\frac{c_2^k}{(b_2-b_1)^{\frac{k}{2}}}+c_2^{\frac{k}{2}}c_3^{\frac{k}{2}}c_7^{\frac{k}{2}}+c_2^{\frac{k}{2}}c_{10}^{\frac{k}{2}}+c_2^\frac{k}{2}c_{11}^\frac{k}{2}\right)\Xi_1\notag\\
&  +\left(c_6^k + c_2^kc_6^k+ c_3^\frac{k}{2} c_6^\frac{k}{2}+c_2^\frac{k}{2} c_3^\frac{k}{2}c_6^\frac{k}{2}+ c_5^k+c_2^kc_5^k\right)\Xi_2\Bigg].
\end{align*}
Consequently, if we set
\[
\overline{M}:= \sup_{t\in (a_2,b_2)}\xi_{W_1}(t,x),
\]
we estimate the last term in the right hand side of \eqref{eq2: Thm: weighted-gradient estimate bounded case} as follows
\begin{align*}
\norm{q\nabla w}_\infty
\leq C\Bigg[&c_2 \overline{M}^\frac{1}{k}
+\left(c_3+\frac{c_2}{(b_2-b_1)^{\frac{1}{2}}}+c_2^{\frac{1}{2}}c_3^{\frac{1}{2}}c_7^{\frac{1}{2}}+c_2^{\frac{1}{2}}c_{10}^{\frac{1}{2}}+c_2^\frac{1}{2}c_{11}^\frac{1}{2}\right)\Xi_1^\frac{1}{k}\notag\\
&  +\left(c_6 + c_2c_6+ c_3^\frac{1}{2} c_6^\frac{1}{2}+c_2^\frac{1}{2} c_3^\frac{1}{2}c_6^\frac{1}{2}+ c_5+c_2c_5\right)\Xi_2^\frac{1}{k}\Bigg]\norm{wq}_\infty^\frac{k-1}{k}.
\end{align*}

Combining \eqref{eq2: Thm: weighted-gradient estimate bounded case} with the above estimates yields
\begin{align*}
\norm{w\nabla q }_\infty
\leq &C c_1^\frac{k}{4} P^\frac{1}{4}\Xi_1^\frac{1}{4} \norm{w\nabla q}_\infty^\frac{1}{2}
+CP^\frac{1}{2k}\left[(c_2+c_7)\Xi_1^\frac{1}{2k}+c_6 \Xi_2^\frac{1}{2k}\right] \norm{w\nabla q}_\infty^\frac{k-1}{k}\\
&+ CP^\frac{1}{k}\left[\left(\frac{c_1}{b_2-b_1}+c_2c_7+c_3+c_4\right)\Xi_1^\frac{1}{k}+(c_2c_6+c_5^2)\Xi_2^\frac{1}{k}\right] \norm{w\nabla q}_\infty^\frac{k-2}{k}
+ C c_2 \Xi_1^\frac{1}{2}\norm{w q}_\infty^\frac{1}{2}\\
&+C\left[\left(\frac{c_2}{b_2-b_1}+c_{11}\right) \Xi_1^\frac{2}{k}+(c_2c_5^2+c_3c_6+c_2c_8+c_9)\Xi_2^\frac{2}{k}\right]\norm{w q}_\infty^\frac{k-2}{k}\\
&+ C\Bigg[c_2 \overline{M}^\frac{1}{k} 
+ \left(c_3+\frac{c_2}{(b_2-b_1)^{\frac{1}{2}}}+c_2^{\frac{1}{2}}c_3^{\frac{1}{2}}c_7^{\frac{1}{2}}+c_2^{\frac{1}{2}}c_{10}^{\frac{1}{2}}+c_2^{\frac{1}{2}}c_{11}^{\frac{1}{2}}\right) \Xi_1^\frac{1}{k}\\
&\quad \quad +\left(c_6 + c_2c_6+ c_3^\frac{1}{2} c_6^\frac{1}{2}+c_2^\frac{1}{2} c_3^\frac{1}{2}c_6^\frac{1}{2}+ c_5+c_2c_5+c_8\right)\Xi_2^\frac{1}{k}\Bigg]\norm{w q}_\infty^\frac{k-1}{k}.
\end{align*}
We observe that, by Young's inequality, we find
\[
C c_1^\frac{k}{4} P^\frac{1}{4}\Xi_1^\frac{1}{4} \norm{w\nabla q}_\infty^\frac{1}{2}
\leq C^2 c_1^\frac{k}{2} P^\frac{1}{2}\Xi_1^\frac{1}{2}+\frac{1}{4}\norm{w\nabla q}_\infty.
\]
Then, setting
\begin{align*}
X:=&\norm{w\nabla q}_\infty^\frac{1}{k},\\
\alpha:=&C^2 c_1^\frac{k}{2} P^\frac{1}{2}\Xi_1^\frac{1}{2} + C c_2 \Xi_1^\frac{1}{2}\norm{w q}_\infty^\frac{1}{2}
+C\Bigg[\left(\frac{c_2}{b_2-b_1}+c_{11}\right)  \Xi_1^\frac{2}{k}\\
&\quad \quad +(c_2c_5^2+c_3c_6+c_2c_8+c_9)\Xi_2^\frac{2}{k}\Bigg]\norm{w q}_\infty^\frac{k-2}{k}\\
&+ C\Bigg[c_2 \overline{M}^\frac{1}{k} 
+ \left(c_3+\frac{c_2}{(b_2-b_1)^{\frac{1}{2}}}+c_2^{\frac{1}{2}}c_3^{\frac{1}{2}}c_7^{\frac{1}{2}}+c_2^{\frac{1}{2}}c_{10}^{\frac{1}{2}}+c_2^{\frac{1}{2}}c_{11}^{\frac{1}{2}}\right) \Xi_1^\frac{1}{k}\\
&\quad \quad +\left(c_6 + c_2c_6+ c_3^\frac{1}{2} c_6^\frac{1}{2}+c_2^\frac{1}{2} c_3^\frac{1}{2}c_6^\frac{1}{2}+ c_5+c_2c_5+c_8\right)\Xi_2^\frac{1}{k}\Bigg]\norm{w q}_\infty^\frac{k-1}{k},\\
\beta:=&CP^\frac{1}{2k}\Big[(c_2+c_7)\Xi_1^\frac{1}{2k}+c_6 \Xi_2^\frac{1}{2k}\Big],\\
\gamma:=&CP^\frac{1}{k}\left[\left(\frac{c_1}{b_2-b_1}+c_2c_7+c_3+c_4\right)\Xi_1^\frac{1}{k}+(c_2c_6+c_5^2)\Xi_2^\frac{1}{k}\right],
\end{align*}
we derive that
\begin{equation}\label{eq1: estimate Dp in case of bounded diff coeff}
X^k \leq \frac{4}{3}\alpha+\frac{4}{3}\beta X^{k-1}+\frac{4}{3}\gamma X^{k-2} .
\end{equation}
We now prove that it leads to
\begin{equation}\label{eq2: estimate Dp in case of bounded diff coeff}
X\leq \frac{4}{3}\beta +\sqrt{\frac{4}{3}\gamma} +\bigg(\frac{4}{3}\alpha\bigg)^\frac{1}{k}.
\end{equation}
We consider the function
\begin{equation*}
f(r):= r^k-\frac{4}{3}\beta r^{k-1}-\frac{4}{3}\gamma r^{k-2}-\frac{4}{3}\alpha= r^{k-2} \left(r^2-\frac{4}{3} \beta r -\frac{4}{3}\gamma\right)-\frac{4}{3}\alpha=: r^{k-2} g(r)-\frac{4}{3}\alpha.
\end{equation*}
First, we show that $f$ is increasing in $\left(\frac{4}{3}\beta+ \sqrt{\frac{4}{3}\gamma}+ (\frac{4}{3}\alpha)^\frac{1}{k}, \infty \right) $. This can be seen by computing the first derivative:
\begin{equation*}
f'(r)=(k-2)r^{k-3}g(r)+r^{k-2} g'(r).
\end{equation*}
Since the function $g$ in positive and increasing in $\left(\frac{4}{3}\beta+ \sqrt{\frac{4}{3}\gamma}+ (\frac{4}{3}\alpha)^\frac{1}{k}, \infty \right) $, it follows that $f'(r)\geq 0$ in the given interval, so $f$ is increasing.\\
Second, we have that
\begin{eqnarray*}
f\left( \frac{4}{3}\beta+ \sqrt{\frac{4}{3}\gamma}+ \left(\frac{4}{3}\alpha\right)^\frac{1}{k}\right)
& = & \left( \frac{4}{3}\beta+ \sqrt{\frac{4}{3}\gamma}+ \left(\frac{4}{3}\alpha\right)^\frac{1}{k}\right)^{k-2} \left[\left( \frac{4}{3}\beta+ \sqrt{\frac{4}{3}\gamma}
+ \left(\frac{4}{3}\alpha\right)^\frac{1}{k}\right)^2 \right.\\
& &- \left.\frac{4}{3} \beta \left( \frac{4}{3}\beta+ \sqrt{\frac{4}{3}\gamma}+ \left(\frac{4}{3}\alpha\right)^\frac{1}{k}\right) - \frac{4}{3} \gamma \right] -\frac{4}{3} \alpha\\
& = &  \left( \frac{4}{3}\beta+ \sqrt{\frac{4}{3}\gamma}+ \left(\frac{4}{3}\alpha\right)^\frac{1}{k}\right)^{k-2} \left[ \left( \frac{4}{3} \alpha\right)^\frac{2}{k}+ \frac{8\sqrt{3}}{9} \sqrt{\gamma} \beta\right. \\
& &\left. + \frac{4\sqrt{3}}{3}  \left(\frac{4}{3}\right)^{\frac{1}{k}} \alpha^\frac{1}{k} \left( \frac{\sqrt{3}}{3} \beta +\sqrt{\gamma}\right)\right]-\frac{4}{3} \alpha\\
&  > & \left( \frac{4}{3} \alpha\right) ^\frac{k-2}{k}\left( \frac{4}{3} \alpha\right)^\frac{2}{k}-\frac{4}{3} \alpha
=0.
\end{eqnarray*}
On one hand, from the previous observations we deduce that $f(r)>0$ if $r>\frac{4}{3}\beta+ \sqrt{\frac{4}{3}\gamma}+ (\frac{4}{3}\alpha)^\frac{1}{k}$. On the other hand, by \eqref{eq1: estimate Dp in case of bounded diff coeff}, $f(X)\leq 0$. Thus, we conclude that \eqref{eq2: estimate Dp in case of bounded diff coeff} holds true. Consequently, there exists a positive constant $K_1$ such that
\[
\norm{w\nabla q}_\infty
\leq K_1 \left(\alpha+\beta^k+\gamma^\frac{k}{2}\right).
\]
By plugging in the previous inequality the definition of $\alpha, \beta,\gamma$ we get
\begin{align*}
\norm{w\nabla q}_{L^\infty(Q(a_2,b_2))}
\leq C\Bigg\lbrace&c_2 \Xi_1^\frac{1}{2}\norm{w q}_{L^\infty(Q(a_2,b_2))}^\frac{1}{2}
+\Bigg[\left(\frac{c_2}{b_2-b_1}+c_{11}\right)  \Xi_1^\frac{2}{k}\\
&+(c_2c_5^2+c_3c_6+c_2c_8+c_9)\Xi_2^\frac{2}{k}\Bigg]\norm{w q}_{L^\infty(Q(a_2,b_2))}^\frac{k-2}{k}\\
&+ \Bigg[c_2 \overline{M}^\frac{1}{k} 
+ \left(c_3+\frac{c_2}{(b_2-b_1)^{\frac{1}{2}}}+c_2^{\frac{1}{2}}c_3^{\frac{1}{2}}c_7^{\frac{1}{2}}+c_2^{\frac{1}{2}}c_{10}^{\frac{1}{2}}+c_2^\frac{1}{2}c_{11}^\frac{1}{2}\right) \Xi_1^\frac{1}{k}\\
&\quad \quad +\left(c_6 + c_2c_6+ c_3^\frac{1}{2} c_6^\frac{1}{2}+c_2^\frac{1}{2} c_3^\frac{1}{2}c_6^\frac{1}{2}+ c_5+c_2c_5+c_8\right)\Xi_2^\frac{1}{k}\Bigg]\norm{w q}_{L^\infty(Q(a_2,b_2))}^\frac{k-1}{k}\\
&+ \Bigg[ \left(c_1^\frac{k}{2}+\frac{c_1^\frac{k}{2}}{(b_2-b_1)^\frac{k}{2}}+c_2^k+c_2^\frac{k}{2}c_7^\frac{k}{2}+c_3^\frac{k}{2} +c_7^k +c_4^\frac{k}{2} \right)\Xi_1^\frac{1}{2}\\
&\quad \quad +(c_6^k+c_2^\frac{k}{2}c_6^\frac{k}{2}+c_5^k)\Xi_2^\frac{1}{2}\Bigg]P^\frac{1}{2}\Bigg\rbrace.
\end{align*}
Letting $a_2\downarrow a$ and $b_2\uparrow b$ and considering that $\int_{a}^{b} \xi_{W_j}(t,x)\,dt\leq \int_{a_0}^{b_0} \xi_{W_j}(t,x)\,dt$ for $j=1,2$, we gain
\begin{align}\label{eq3: Thm: weighted-gradient estimate bounded case}
|w(t,&y)  \nabla p(t,x,y)|
\leq C\Bigg\lbrace c_2 \Xi_1(a_0,b_0)^\frac{1}{2}\norm{w p}_{L^\infty(Q(a,b))}^\frac{1}{2}
+\Bigg[\left(\frac{c_2}{b-b_1}+c_{11}\right) \Xi_1(a_0,b_0)^\frac{2}{k}\notag\\
&+(c_2c_5^2+c_3c_6+c_2c_8+c_9)\Xi_2(a_0,b_0)^\frac{2}{k}\Bigg]\norm{w p}_{L^\infty(Q(a,b))}^\frac{k-2}{k}
+ \Bigg[c_2 \left(\sup_{t\in (a_0,b_0)}\xi_{W_1}(t,x)\right)^\frac{1}{k}\notag\\
& + \Bigg(c_3+\frac{c_2}{(b-b_1)^{\frac{1}{2}}}+c_2^{\frac{1}{2}}c_3^{\frac{1}{2}}c_7^{\frac{1}{2}}+c_2^{\frac{1}{2}}c_{10}^{\frac{1}{2}}+c_2^\frac{1}{2}c_{11}^\frac{1}{2}\Bigg)\Xi_1(a_0,b_0)^\frac{1}{k}+\Big(c_6 + c_2c_6+ c_3^\frac{1}{2} c_6^\frac{1}{2}+c_2^\frac{1}{2} c_3^\frac{1}{2}c_6^\frac{1}{2}+ c_5\notag\\
& +c_2c_5+c_8\Big)\Xi_2(a_0,b_0)^\frac{1}{k}\Bigg]\norm{w p}_{L^\infty(Q(a,b))}^\frac{k-1}{k}
+ \Bigg[ \left(c_1^\frac{k}{2}+\frac{c_1^\frac{k}{2}}{(b-b_1)^\frac{k}{2}}+c_2^k+c_2^\frac{k}{2}c_7^\frac{k}{2}+c_3^\frac{k}{2} +c_7^k +c_4^\frac{k}{2}\right)\times\notag\\
&\times \Xi_1(a_0,b_0)^\frac{1}{2}+(c_6^k+c_2^\frac{k}{2}c_6^\frac{k}{2}+c_5^k)\Xi_2(a_0,b_0)^\frac{1}{2}\Bigg]\left(\int_{Q(a,b)}\frac{\abs{\nabla p}^2}{p}\,dt\, dy\right)^\frac{1}{2},
\end{align}
for all $(t,y)\in (a_1,b_1)\times \R^d$ and fixed $x\in\R^d$.

To finish the proof, it remains to remove the additional assumption on the weight $w$.
We define
\[
w_\varepsilon:= \frac{w}{1+\varepsilon w}, \quad \varepsilon >0.
\]
Since 
\[
D_i w_\varepsilon=(1+\varepsilon w)^{-2} D_i w
\]
and
\[
D_{ij}w_\varepsilon=(1+\varepsilon w)^{-2} D_{ij}w-2\varepsilon (1+\varepsilon w)^{-3} (D_i w)( D_j w),
\]
for all $i,j=1, \dots, d$, then by Hypothesis \ref{h.3}(b) it follows that $w_\varepsilon$, along with its first order partial derivatives and its second order partial derivatives of the form $D_{ij}w$ and $\partial_t D_i w$ are bounded.
If we now check Hypothesis \ref{h.3}(c) we have that

\[
w_\varepsilon\leq w\leq c_1^\frac{k}{2} W_1^\frac{1}{2},
\]

\[
|Q\nabla w_\varepsilon|= (1+\varepsilon w)^{-2} |Q\nabla w|
\leq c_2 W_1^\frac{1}{2k},
\]

\[
|QD^2 w_\varepsilon|
\leq (1+\varepsilon w)^{-2}|QD^2 w|+2\varepsilon (1+\varepsilon w)^{-3} |Q\nabla w| |\nabla w|
\leq (c_3 + 2\eta^{-1}c_2^2)W_1^\frac{1}{k},
\]

\begin{align*}
|D^3 w_\varepsilon|
&\leq (1+\varepsilon w)^{-2} |D^3 w|+6\varepsilon(1+\varepsilon w)^{-3}|\nabla w| |D^2 w|+6\varepsilon^2 (1+\varepsilon w)^{-4} |\nabla w|^3\\
&\leq (c_{10} +6\eta^{-2}c_2c_3+6\eta^{-3}c_2^3) W_1^\frac{3}{2k},
\end{align*}

\[
|\partial_t w_\varepsilon|= (1+\varepsilon w)^{-2} |\partial_t w|
\leq c_4 w^\frac{k-2}{k} W_1^\frac{1}{2k}
\]
and
\[
|\partial_t \nabla w_\varepsilon|
\leq (1+\varepsilon w)^{-2} |\partial_t \nabla w|+2\varepsilon (1+\varepsilon w)^{-3} |\nabla w| |\partial_t w|
\leq (c_{11}+2\eta^{-1}c_2 c_4) W_1^\frac{1}{k}.
\]

This shows that $w_\varepsilon$ satisfies Hypothesis \ref{h.3}(c) with the same constants $c_1, c_2, c_4, c_5, c_6, c_7, c_8, c_9$ and with the constants $c_3$, $c_{10}$, $c_{11}$ replaced, respectively, by $c_3+2\eta^{-1}c_2^2$, $c_{10}+6\eta^{-2}c_2c_3+6\eta^{-3}c_2^3$ and $c_{11}+2\eta^{-1}c_2 c_4$.

Thus, the estimate \eqref{eq3: Thm: weighted-gradient estimate bounded case} shown in the first part of the proof holds true with the function $w_\varepsilon$ instead of $w$ and with the constants on the right hand side that do not depend on $\varepsilon$.
We finally let $\varepsilon\to 0$ to gain the desired inequality \eqref{weighted-gradient estimate bounded case}.
\end{proof}

\begin{rem}
From the above proof one can see that the constants $B_i,\,i=1,\ldots ,6$ are given by
\begin{align}\label{def: constants B_i}
B_1&=c_2,\notag\\
B_2&=\frac{c_2}{b-b_1}+c_2c_4+c_{11},\notag \\
B_3&=c_2c_5^2+c_3c_6+c_2^2c_6+c_2c_8+c_9,\notag\\
B_4&= c_3+c_2^2+\frac{c_2}{(b-b_1)^{\frac{1}{2}}}+c_2^{\frac{1}{2}}c_3^{\frac{1}{2}}c_7^{\frac{1}{2}}+c_2^{\frac{3}{2}}c_7^{\frac{1}{2}}+c_2^{\frac{1}{2}}c_{10}^{\frac{1}{2}}+c_2c_3^{\frac{1}{2}}+c_2c_4^\frac{1}{2}+c_2^\frac{1}{2}c_{11}^\frac{1}{2}, \notag\\
B_5&= c_6 + c_2c_6+ c_3^\frac{1}{2} c_6^\frac{1}{2}+c_2c_6^{\frac{1}{2}}+c_2^\frac{1}{2} c_3^\frac{1}{2}c_6^\frac{1}{2}+c_2^{\frac{3}{2}}c_6^{\frac{1}{2}}+ c_5+c_2c_5+c_8, \notag\\
B_6&= c_1^\frac{k}{2}+\frac{c_1^\frac{k}{2}}{(b-b_1)^\frac{k}{2}} +c_2^k+c_2^\frac{k}{2}c_7^\frac{k}{2}+c_3^\frac{k}{2} +c_7^k +c_4^\frac{k}{2}, \notag\\
B_7&= c_6^k+c_2^\frac{k}{2}c_6^\frac{k}{2}+c_5^k.
\end{align}
\end{rem}

\section{Results for general diffusion coefficients}

\subsection{Approximation of the coefficients}\label{s.approx}

We approximate the operator $A$ with a family of operators $A_n$ with bounded diffusion coefficients in order to apply the results from the previous section.

To that end, we approximate the diffusion matrix $Q$ as follows. Picking a function $\varphi\in C_c^\infty(\R)$ such that $0\le \varphi\le 1,\,\varphi\equiv 1$ in $(-1,1)$, $\varphi\equiv 0$ in $\R\setminus (-2,2)$ and $\abs{s\varphi'(s)}\leq 2$ for all $s\in\R$, we set
\begin{equation*}
\varphi_n(x):= \varphi(W_1(t_0,x)/n),
\end{equation*}
where $W_1$ is the Lyapunov function from Hypothesis \ref{h.3}. The constant $t_0\in (0,T)$ will be chosen later on.
We then put
\begin{equation*}
q_{ij}^{(n)}(x):= \varphi_n(x)q_{ij}(x)+(1-\varphi_n(x))\eta \delta_{ij},
\end{equation*}
where $\delta_{ij}$ is the Kronecker delta. Replacing $Q$ with $Q_n:=(q_{ij}^{(n)})$ we approximate $A$ with the operators $A_n$ defined by
\begin{equation*}
A_n = {\rm div}(Q_n\nabla )+F\cdot \nabla-V.
\end{equation*}

\begin{lemma}\label{Lemma: A_n satisfies h.1}
For every $n\in\N$ the diffusion coefficients $q_{ij}^{(n)}$ and their first order spatial derivatives are bounded on $\R^d$. Moreover, the operator $A_n$ satisfies Hypothesis \ref{h.1} and if $W$ is a time dependent Lyapunov function for the operator $\partial_t+A$ such that $|\nabla W|$ is bounded on $[0,T]\times B_R$ for all $R>0$, then $W$ is a time dependent Lyapunov function for $\partial_t+A_n$.
\end{lemma}

\begin{proof}
Clearly, since $\lim_{\abs{x}\to\infty} W_1(t_0,x)=+\infty$, the functions $\varphi_n$ vanish outside a compact set. As a consequence, the coefficients $q_{ij}^{(n)}$ and their spatial derivatives $D_k q_{ij}^{(n)}$ are bounded on $\R^d$ for all $i,j,k=1, \dots, d$.
We now check Hypothesis \ref{h.1}. First, we observe that $Q_n$ is symmetric and, thanks to the uniformly ellipticity of $Q$, we get
\[
\sum_{i,j=1}^d q_{ij}^{(n)}(x)\xi_i\xi_j
= \varphi_n(x) \sum_{i,j=1}^d q_{ij}(x)\xi_i\xi_j + \eta (1-\varphi_n(x)) \abs{\xi}^2
\ge \eta |\xi|^2 
\]
for all $x,\,\xi\in \R^d$. It remains to prove that $A_nZ(x)\le M_1$ for a certain constant $M_1\geq 0$ and for all $x\in\R^d$. Let $x\in\R^d$. Then
\begin{align*}
A_nZ(x)
=& {\rm div}(Q_n\nabla Z(x))+F(x)\cdot \nabla Z(x)-V(x)Z(x)\\
=& \varphi_n(x) {\rm div}(Q\nabla Z(x)) + Q \nabla \varphi_n(x)\cdot \nabla Z(x)-\eta \nabla \varphi_n(x)\cdot \nabla Z(x)\\
&+\eta (1- \varphi_n(x))\Delta Z(x)+F(x)\cdot \nabla Z(x) -V(x)Z(x)\\
=& \varphi_n(x) AZ(x) +(1- \varphi_n(x))(\eta\Delta Z(x)+F(x)\cdot \nabla Z(x)-V(x)Z(x)) \\
&+ Q \nabla \varphi_n(x)\cdot \nabla Z(x)-\eta \nabla \varphi_n(x)\cdot \nabla Z(x).
\end{align*}
For the first and the second term in the right hand side we apply Hypothesis \ref{h.1} that holds true for the operator $A$:
$$A_nZ(x)\leq M + Q \nabla \varphi_n(x)\cdot \nabla Z(x)-\eta \nabla \varphi_n(x)\cdot \nabla Z(x).$$
We can find a bound also for the last two terms since the functions $\varphi_n$ vanish outside a compact set. 
As a result we find a constant $M_1\ge 0$ such that 
$$A_nZ(x)\leq M_1\quad \forall x\in\R^d.$$
Finally, we check that if $W$ is a time dependent Lyapunov function for the operator $\partial_t+A$ such that $|\nabla W|$ is bounded on $[0,T]\times B_R$ for all $R>0$, then $W$ is a time dependent Lyapunov function for $\partial_t+A_n$. 
This can be seen by computing $\partial_t W(t,y)+A_n W(t,y)$ for $(t,y)\in (0,T)\times\R^d$:
\begin{align*}
\partial_t W(t,y)+A_n W(t,y)
=& \partial_t W(t,y) + {\rm div}(Q_n\nabla W(t,y))+F(y) \cdot \nabla W(t,y)-V(y)W(t,y)\\
=&  \varphi_n(y) LW(t,y) +(1- \varphi_n(y))(\partial_t W(t,y)+\eta\Delta W(t,y)+F(y) \cdot \nabla W(t,y)\\
&-V(y)W(t,y)) + Q \nabla \varphi_n(y)\cdot \nabla W(t,y)-\eta \nabla \varphi_n(y)\cdot \nabla W(t,y).
\end{align*}
Since $W$ is a time dependent Lyapunov function  for the operator $\partial_t+A$, the first two terms in the right hand side are bounded by $h(t)W(t,y)$: 
$$\partial_t W(t,y)+A_n W(t,y)\leq h(t)W(t,y)+ Q \nabla \varphi_n(y)\cdot \nabla W(t,y)-\eta \nabla \varphi_n(y)\cdot \nabla W(t,y) ,$$
where $h(t)\in L^1((0,T))$.
Furthermore, the last terms are bounded by a nonnegative constant because $\varphi_n$ vanishes outside a compact set and $|\nabla W|$ is bounded on $[0,T]\times B_R$ for all $R>0$.
Hence, there is a function $\tilde{h}(t)\in L^1((0,T))$ such that
$$\partial_t W(t,y)+A_n W(t,y)\leq \tilde{h}(t) W(t,y)$$
for all $(t,y)\in (0,T)\times\R^d$.
Then $W$ is a time dependent Lyapunov function also for $\partial_t+A_n$.
\end{proof}

As a consequence of the previous lemma, for every $n\in\N$ the semigroup generated by $A_n$ in $C_b(\R^d)$ is given by a kernel $p_n(t,x,y)$.
In order to show further properties about the operators $A_n$, we make the following assumptions.

\begin{hyp}\label{h.4}
Fix $T>0$, $x\in \R^d$ and $0<a_0<a<b<b_0<T$. Let us consider two time dependent Lyapunov functions $W_1,\,W_2$ with $W_1\le W_2$ and $|\nabla W_1|,|\nabla W_2|$ bounded on $[0,T]\times B_R$ for all $R>0$ and a weight function $1\le w\in C^{1,2}((0,T)\times \R^d)$ such that
\begin{enumerate}
\item there is $t_0\in (0,T)$ such that
\[
|Q| |\nabla W_1(t_0,\cdot)|\le c_{12} W_1(t_0, \cdot) w^{-1/k}W_1^{1/2k};
\]
\item there are $c_0>0$ and $\sigma\in (0,1)$ such that
$$W_2\leq c_0 Z^{1-\sigma}.$$
\end{enumerate}
\end{hyp}

We now prove that if the operator $A$ satisfies Hypothesis \ref{h.3}, then the same is true for the operators $A_n$ assuming further Hypothesis \ref{h.4}.

\begin{lemma}\label{Lemma: A_n satisfies h.3}
Assume that the operator $A$ satisfies Hypotheses \ref{h.3}(c) and \ref{h.4}(a). Then the operator $A_n$ satisfies Hypothesis \ref{h.3}(c) with the same constants $c_1, c_4, c_5, c_6, c_8, c_9, c_{10}, c_{11}$ and with $c_2, c_3, c_7$ being replaced, respectively, by $2 c_2, 2c_3$ and $\sqrt{3}(c_7+2 (1+\sqrt{d}) c_{12})$.
\end{lemma}

\begin{proof}
The constants $c_1, c_4, c_5, c_6, c_8, c_9, c_{10}, c_{11}$ remain the same because the corrispondent inequalities do not depend on the diffusion coefficients. Let us note that Hypothesis \ref{h.3}(c)-(ii) implies that
$$\abs{\nabla w}\leq \eta^{-1}c_2 W_1^\frac{1}{2k}.$$
So, it follows that
\[
\abs{Q_n \nabla w}
= \abs{\varphi_n Q \nabla w+(1-\varphi_n)\eta \nabla w}
\leq \abs{Q \nabla w} +\eta \abs{\nabla w}
\leq 2 c_2 W_1^\frac{1}{2k}.
\]
Similarly, we get
\[
\abs{Q_n D^2 w}
= \abs{\varphi_n Q D^2 w+(1-\varphi_n)\eta D^2 w}
\leq \abs{Q D^2 w} +\eta \abs{D^2 w}
\leq 2 c_3 W_1^\frac{1}{k}.
\]
Finally, for $(t,y)\in [a_0,b_0]\times\R^d$, given that $\abs{s\varphi'(s)}\leq 2$ and using Hypothesis \ref{h.4}(a), we have
\begin{align*}
|\nabla & Q_n(t,y)|^2
= \sum_{i,j,h=1}^d \abs{\varphi_n D_h q_{ij}+\frac{\varphi'(W_1(t_0,y)/n)}{n} D_h W_1(t_0,y)(q_{ij}-\eta \delta_{ij}) }^2\\
\leq& 3\sum_{i,j,h=1}^d \left[ |\varphi_n D_h q_{ij}|^2+ \frac{|\varphi'(W_1(t_0,y)/n)|^2}{n^2} |D_h W_1(t_0,y)|^2 (q_{ij}^2+\eta^2 \delta_{ij})  \right]\\
\leq& 3 |\varphi_n|^2 |\nabla Q|^2 + 3 (W_1(t_0,y)/n)^2 |\varphi'(W_1(t_0,y)/n)|^2 (W_1(t_0,y))^{-2}\sum_{i,j,h=1}^d |q_{ij}D_h W_1(t_0,y)|^2\\
&+3 (W_1(t_0,y)/n)^2 |\varphi'(W_1(t_0,y)/n)|^2 (W_1(t_0,y))^{-2}\sum_{h=1}^d |D_h W_1(t_0,y)|^2 \eta^2 \sum_{i,j=1}^d \delta_{ij}\\
\leq & 3(c_7^2 +4c_{12}^2+4dc_{12}^2) w^{-\frac{2}{k}} W_1^\frac{1}{k}.
\end{align*}
Then,
\[
|\nabla Q_n(t,y)|
\leq \sqrt{3}(c_7 +2 c_{12}+2\sqrt{d} c_{12}) w^{-\frac{1}{k}} W_1^\frac{1}{2k}.\qedhere
\]
\end{proof}

We can now obtain estimates for the gradients of the kernels $p_n$.

\begin{lemma}\label{Lemma: weighted-gradient estimate for p_n}
Assume Hypotheses \ref{h.4} hold and that the operator $A$ satisfies Hypotheses \ref{h.3}.
For $i=1,2$, we set
\[
\xi_{W_i,n} (t,x):=\int_{\R^d}p_n(t,x,y)W_i(t,y)\,dy
\quad \text{ and }\quad 
\Xi_{i,n}(a_0,b_0):= \int_{a_0}^{b_0}\xi_{W_i,n} (t,x)\, dt.
\]
Then for any $n\in \N$ we have
\begin{equation*}\label{weighted-gradient estimate for p_n in short}
|w(t,y)\nabla p_n(t,x,y)| \leq K_n
\end{equation*}
for all $(t,y)\in (a_1,b_1)\times \R^d$ and fixed $x\in\R^d$, where
\begin{align*}\label{RHS weighted-gradient estimate for p_n}
K_n=& C\Bigg\lbrace B_1 A_1^\frac{k-1}{k}\sup_{t\in (a_0,b_0)}\xi_{W_1,n}(t,x) +\left( B_1 \tilde{A}_2^\frac{1}{2} +B_2 \tilde{A}_2^\frac{k-2}{k}+\tilde{B}_4 \tilde{A}_2^\frac{k-1}{k}\right)\Xi_{1,n}(a_0,b_0)\notag\\
&+ \Big[B_1 A_3^\frac{1}{2}+(B_2+B_3) A_3^\frac{k-2}{k}+ B_3\tilde{A}_2^\frac{k-2}{k} + (\tilde{B}_4+B_5) A_3^\frac{k-1}{k}+B_5 \tilde{A}_2^\frac{k-1}{k}\notag\\
&+\tilde{B}_6 B_8+B_7B_8\Big]\Xi_{2,n}(a_0,b_0)
+ B_1 A_1^\frac{1}{2}\,\Xi_{1,n}(a_0,b_0)^\frac{1}{2}\left(\sup_{t\in (a_0,b_0)}\xi_{W_1,n}(t,x)\right)^\frac{1}{2}\notag\\
&+A_1^\frac{k-2}{k}\left(B_2\,\Xi_{1,n}(a_0,b_0)^\frac{2}{k}+B_3\,\Xi_{2,n}(a_0,b_0)^\frac{2}{k} \right)\left(\sup_{t\in (a_0,b_0)}\xi_{W_1,n}(t,x)\right)^\frac{k-2}{k}\notag\\
&+B_1\left(\tilde{A}_2^\frac{k-1}{k}\,\Xi_{1,n}(a_0,b_0)^\frac{k-1}{k}+A_3^\frac{k-1}{k}\,\Xi_{2,n}(a_0,b_0)^\frac{k-1}{k} \right)\left(\sup_{t\in (a_0,b_0)}\xi_{W_1,n}(t,x)\right)^\frac{1}{k}\notag\\
&+A_1^\frac{k-1}{k}\left(\tilde{B}_4\,\Xi_{1,n}(a_0,b_0)^\frac{1}{k} +B_5\, \Xi_{2,n}(a_0,b_0)^\frac{1}{k}\right)\left(\sup_{t\in (a_0,b_0)}\xi_{W_1,n}(t,x)\right)^\frac{k-1}{k}\\
&+ \left(\int_{Q(a,b)} p_n(t,x,y)\log^2 p_n(t,x,y) \,dt\,dy\right)^\frac{1}{2}
 - \left(\int_{\R^d} [p_n(t,x,y)\log p_n(t,x,y)]_{t=a}^{t=b}dy\right)^\frac{1}{2}\Bigg]\Bigg\rbrace,
\end{align*}
and the constants $A_1, A_3, B_1, \dots, B_8, \tilde{A}_2, \tilde{B}_4, \tilde{B}_6$ are defined as in \eqref{def: constants A_i}, \eqref{def: constants B_i}, \eqref{def: constant B_8}, \eqref{def: constants tilde A_2} and \eqref{def: constants tilde B_i}.
\end{lemma}

\begin{proof}
Since the operator $A$ satisfies Hypotheses \ref{h.3} and \ref{h.4}, then for any $n\in \N$ the operator $A_n$ satisfies Hypotheses \ref{h.1} and \ref{h.3} with slightly different constants by Lemmas \ref{Lemma: A_n satisfies h.1} and \ref{Lemma: A_n satisfies h.3}.
Consequently, applying \eqref{Estimate for wp in the bounded case under h.3} to $p_n$ we get
\begin{equation}\label{eq.pnestimate}
w(t,y)p_n(t,x,y) \leq C\left(A_1\sup_{t\in (a_0,b_0)}\xi_{W_1,n}(t,x)+\tilde{A}_2\, \Xi_{1,n}(a_0,b_0) +A_3\,\Xi_{2,n}(a_0,b_0)\right),
\end{equation}
where
\begin{equation}\label{def: constants tilde A_2}
\tilde{A}_2 =c_2^k+\frac{c_1^{\frac{k}{2}}}{(b_0-b)^{\frac{k}{2}}}+c_2^{\frac{k}{2}}c_7^{\frac{k}{2}}+c_2^{\frac{k}{2}}c_{12}^{\frac{k}{2}}+c_3^{\frac{k}{2}} +c_4^\frac{k}{2}.
\end{equation}
Moreover, applying \eqref{weighted-gradient estimate bounded case} to $p_n$, we obtain

\begin{align*}
|w(t,y)&\nabla p_n(t,x,y)|
\leq C\Bigg\lbrace B_1 \,\Xi_{1,n}(a_0,b_0)^\frac{1}{2}\norm{w p_n}_{L^\infty(Q(a,b))}^\frac{1}{2}\\
&+\left(B_2 \, \Xi_{1,n}(a_0,b_0)^\frac{2}{k}+B_3\,\Xi_{2,n}(a_0,b_0)^\frac{2}{k}\right)\norm{w p_n}_{L^\infty(Q(a,b))}^\frac{k-2}{k}\notag\\
&+ \Bigg[B_1\left(\sup_{t\in (a_0,b_0)}\xi_{W_1,n}(t,x)\right)^\frac{1}{k} 
+ \tilde{B}_4\, \Xi_{1,n}(a_0,b_0)^\frac{1}{k}+B_5\,\Xi_{2,n}(a_0,b_0)^\frac{1}{k}\Bigg]\norm{w p_n}_{L^\infty(Q(a,b))}^\frac{k-1}{k}\notag\\
&+ \left( \tilde{B}_6 \,\Xi_{1,n}(a_0,b_0)^\frac{1}{2} +B_7\,\Xi_{2,n}(a_0,b_0)^\frac{1}{2}\right)\left(\int_{Q(a,b)}\frac{\abs{\nabla p_n}^2}{p_n}\,dt\, dy\right)^\frac{1}{2}\Bigg\rbrace,
\end{align*}
 where
\begin{align}\label{def: constants tilde B_i}
\tilde{B}_4&= c_3+c_2^2+\frac{c_2}{(b-b_1)^{\frac{1}{2}}}+c_2^{\frac{1}{2}}c_3^{\frac{1}{2}}c_7^{\frac{1}{2}}
+c_2^{\frac{1}{2}}c_3^{\frac{1}{2}}c_{12}^{\frac{1}{2}}+c_2^{\frac{3}{2}}c_7^{\frac{1}{2}}+c_2^{\frac{3}{2}}c_{12}^{\frac{1}{2}}+c_2^{\frac{1}{2}}c_{10}^{\frac{1}{2}}+c_2c_3^{\frac{1}{2}}+c_2c_4^\frac{1}{2}+c_2^\frac{1}{2}c_{11}^\frac{1}{2}, \notag \\
\tilde{B}_6&= c_1^\frac{k}{2}+\frac{c_1^\frac{k}{2}}{(b-b_1)^\frac{k}{2}} +c_2^k+c_2^\frac{k}{2}c_7^\frac{k}{2}+c_2^\frac{k}{2}c_{12}^\frac{k}{2}+c_3^\frac{k}{2} +c_7^k+c_{12}^k +c_4^\frac{k}{2}.
\end{align}
Finally, by Theorem \ref{Thm: similar to Thm 5.1,Lem 5.1, Thm 5.2 MPR}(a) we have
\begin{align*}
\int_{Q(a,b)}\frac{\abs{\nabla p_n}^2}{p_n}\,dt\, dy
\leq \Bigg[&(c_6^2+c_5^4)\,\Xi_{2,n}(a_0,b_0)+ \int_{Q(a,b)} p_n(t,x,y)\log^2 p_n(t,x,y) \,dt\,dy\notag\\
& - \int_{\R^d} [p_n(t,x,y)\log p_n(t,x,y)]_{t=a}^{t=b}dy\Bigg].
\end{align*}
Combining them yields
\begin{align*}
|w(t,y)&\nabla p_n(t,x,y)|
\leq C\Bigg\lbrace B_1 \,\Xi_{1,n}(a_0,b_0)^\frac{1}{2}\Bigg[A_1^\frac{1}{2}\left(\sup_{t\in (a_0,b_0)}\xi_{W_1,n}(t,x)\right)^\frac{1}{2}+\tilde{A}_2^\frac{1}{2}\, \Xi_{1,n}(a_0,b_0)^\frac{1}{2}\\ 
&+A_3^\frac{1}{2}\,\Xi_{2,n}(a_0,b_0)^\frac{1}{2}\Bigg]
+\left(B_2 \, \Xi_{1,n}(a_0,b_0)^\frac{2}{k}+B_3\,\Xi_{2,n}(a_0,b_0)^\frac{2}{k}\right)\Bigg[ A_1^\frac{k-2}{k}\left(\sup_{t\in (a_0,b_0)}\xi_{W_1,n}(t,x)\right)^\frac{k-2}{k}\\
&+\tilde{A}_2^\frac{k-2}{k}\, \Xi_{1,n}(a_0,b_0)^\frac{k-2}{k} +A_3^\frac{k-2}{k}\,\Xi_{2,n}(a_0,b_0)^\frac{k-2}{k}\Bigg]+ \Bigg[B_1\left(\sup_{t\in (a_0,b_0)}\xi_{W_1,n}(t,x)\right)^\frac{1}{k}\notag\\
& + \tilde{B}_4\, \Xi_{1,n}(a_0,b_0)^\frac{1}{k}+B_5\,\Xi_{2,n}(a_0,b_0)^\frac{1}{k}\Bigg]\Bigg[A_1^\frac{k-1}{k}\left(\sup_{t\in (a_0,b_0)}\xi_{W_1,n}(t,x)\right)^\frac{k-1}{k}+\tilde{A}_2^\frac{k-1}{k}\, \Xi_{1,n}(a_0,b_0)^\frac{k-1}{k}\\
& +A_3^\frac{k-1}{k}\,\Xi_{2,n}(a_0,b_0)^\frac{k-1}{k}\Bigg]+ \left( \tilde{B}_6 \,\Xi_{1,n}(a_0,b_0)^\frac{1}{2} +B_7\,\Xi_{2,n}(a_0,b_0)^\frac{1}{2}\right)\Bigg[B_8\,\Xi_{2,n}(a_0,b_0)^\frac{1}{2}\\
&+ \left(\int_{Q(a,b)} p_n(t,x,y)\log^2 p_n(t,x,y) \,dt\,dy\right)^\frac{1}{2}
 - \left(\int_{\R^d} [p_n(t,x,y)\log p_n(t,x,y)]_{t=a}^{t=b}dy\right)^\frac{1}{2}\Bigg]\Bigg\rbrace,
\end{align*}
where 
\begin{equation}\label{def: constant B_8}
B_8=c_6+c_5^2.
\end{equation}
Considering that $\Xi_{1,n}(a_0,b_0)\leq \Xi_{2,n}(a_0,b_0)$, we gain the desired estimate.
\end{proof}

\subsection{Estimates for the derivatives of the kernel}\label{s.generalest}

\begin{lemma}\label{Lemma: n goes to infinity}
Assume Hypotheses \ref{h.4} hold and that the operator $A$ satisfies Hypotheses \ref{h.3}. Then the following statements hold.
\begin{enumerate}
\item $p_n(t, x, \cdot)\to p(t,x,\cdot)$ locally uniformly in $\R^d$ as $n\to\infty$.
\item $ \xi_{W_j,n}(\cdot,x)\to  \xi_{W_j}(\cdot,x)$ uniformly in $(a_0,b_0)$ as $n\to\infty$ for $j=1,2$.
\item $\int_{\R^d}[p_n(t,x,y)\log p_n(t,x,y)]_{t=a}^{t=b}dy \to \int_{\R^d}[p(t,x,y)\log p(t,x,y)]_{t=a}^{t=b}dy$ as $n\to\infty$
and 
 $\int_{Q(a,b)} p_n(t,x,y)\log^2 p_n(t,x,y) \,dt\,dy \to \int_{Q(a,b)} p(t,x,y)\log^2 p(t,x,y) \,dt\,dy$ as $n\to\infty$.
 In particular, the latter integrals are finite.
\end{enumerate} 
\end{lemma}

\begin{proof}
(a) follows as in \cite[Lemma 12.7]{KunzeLorenziRhandi1} and (b) as in the proof of \cite[Theorem 12.6]{KunzeLorenziRhandi1}.
It follows from Equation \eqref{eq.pnestimate}, that $p_n \leq C_n w^{-1}$ for a certain constant $C_n$. But by (a) and (b) $\sup C_n <\infty$.
Making use of Hypothesis \ref{h.3}(a), we find integrable majorants for $p_n\log p_n$ and $p_n\log^2p_n$. At this point, (c) follows from (a) by means of  the dominated convergence theorem.
\end{proof}

\begin{cor}\label{c.logpinl2}
Assume Hypotheses \ref{h.3} and \ref{h.4}. Then $\sqrt{p}\in W^{0,1}_2(Q(a,b))$.
\end{cor}

\begin{proof}
As a consequence of Hypothesis \ref{h.3}(b) and Lemma \ref{Lemma: n goes to infinity}, 
\begin{align*}
C & := \sup_{n\in \N} \frac{1}{\eta^2}\int_{Q(a,b)} \left(|F(y)|^2+V^2(y)\right) p_n(t,x,y)\,dt\,dy\\
& + \int_{Q(a,b)} p_n(t,x,y)\log^2 p_n(t,x,y) \,dt\,dy\notag\\
&- \frac{2}{\eta} \int_{\R^d} [p_n(t,x,y)\log p_n(t,x,y)]_{t=a}^{t=b}dy<\infty.
\end{align*}
It follows from Theorem \ref{Thm: similar to Thm 5.1,Lem 5.1, Thm 5.2 MPR} that $\sqrt{p_n}$ is bounded in $W^{0,1}_2(Q(a,b))$. As this space is reflexive, a subsequence of $p_n$ converges weakly to some element $q$ of $W^{0,1}_2(Q(a,b))$. However, as $p_n \to p$ pointwise and with an integrable majorant, testing against a function in $C_c^\infty(\R^d)$, we see that $q=p$.
\end{proof}

We can now prove our main result.

\begin{theo}\label{Thm: weighted-gradient estimate}
Assume that the operator $A$ satisfies Hypotheses \ref{h.3} and \ref{h.4}. Then we have
\begin{align}\label{Estimate for Dp}
|w(t,y)\nabla p(t,x,y)|\leq K,
\end{align}
for all $(t,y)\in (a,b)\times \R^d$ and fixed $x\in\R^d$,
where
\begin{align}\label{RHS weighted-gradient estimate for p}
K= & C\Bigg\lbrace B_1 A_1^\frac{k-1}{k}\sup_{t\in (a_0,b_0)}\xi_{W_1}(t,x) +\left( B_1 \tilde{A}_2^\frac{1}{2} +B_2 \tilde{A}_2^\frac{k-2}{k}+\tilde{B}_4 \tilde{A}_2^\frac{k-1}{k}\right)\Xi_{1}(a_0,b_0)\notag\\
&+ \Big[B_1 A_3^\frac{1}{2}+(B_2+B_3) A_3^\frac{k-2}{k}+ B_3\tilde{A}_2^\frac{k-2}{k} + (\tilde{B}_4+B_5) A_3^\frac{k-1}{k}+B_5 \tilde{A}_2^\frac{k-1}{k}\notag\\
&+\tilde{B}_6 B_8+B_7B_8\Big]\Xi_{2}(a_0,b_0)
+ B_1 A_1^\frac{1}{2}\,\Xi_{1}(a_0,b_0)^\frac{1}{2}\left(\sup_{t\in (a_0,b_0)}\xi_{W_1}(t,x)\right)^\frac{1}{2}\notag\\
&+A_1^\frac{k-2}{k}\left(B_2\,\Xi_{1}(a_0,b_0)^\frac{2}{k}+B_3\,\Xi_{2}(a_0,b_0)^\frac{2}{k} \right)\left(\sup_{t\in (a_0,b_0)}\xi_{W_1}(t,x)\right)^\frac{k-2}{k}\notag\\
&+B_1\left(\tilde{A}_2^\frac{k-1}{k}\,\Xi_{1}(a_0,b_0)^\frac{k-1}{k}+A_3^\frac{k-1}{k}\,\Xi_{2}(a_0,b_0)^\frac{k-1}{k} \right)\left(\sup_{t\in (a_0,b_0)}\xi_{W_1}(t,x)\right)^\frac{1}{k}\notag\\
&+A_1^\frac{k-1}{k}\left(\tilde{B}_4\,\Xi_{1}(a_0,b_0)^\frac{1}{k} +B_5\, \Xi_{2}(a_0,b_0)^\frac{1}{k}\right)\left(\sup_{t\in (a_0,b_0)}\xi_{W_1}(t,x)\right)^\frac{k-1}{k}\notag\\
&+ \left( \tilde{B}_6 \,\Xi_{1}(a_0,b_0)^\frac{1}{2} +B_7\,\Xi_{2}(a_0,b_0)^\frac{1}{2}\right)\left(\int_{Q(a,b)} p(t,x,y)\log^2 p(t,x,y) \,dt\,dy\right)^\frac{1}{2}\notag\\
& - \left( \tilde{B}_6 \,\Xi_{1}(a_0,b_0)^\frac{1}{2} +B_7\,\Xi_{2}(a_0,b_0)^\frac{1}{2}\right)\left(\int_{\R^d} [p(t,x,y)\log p(t,x,y)]_{t=a}^{t=b}dy\right)^\frac{1}{2}\Bigg]\Bigg\rbrace,
\end{align}
and the constants $A_1, A_3, B_1, \dots, B_8, \tilde{A}_2, \tilde{B}_4, \tilde{B}_6$ are defined as in \eqref{def: constants A_i}, \eqref{def: constants B_i}, \eqref{def: constant B_8}, \eqref{def: constants tilde A_2} and \eqref{def: constants tilde B_i}.
\end{theo}

\begin{proof}
By Lemmas \ref{Lemma: weighted-gradient estimate for p_n} and \ref{Lemma: n goes to infinity} we infer that
\[
\limsup _{n\to\infty }|w(t,y)\nabla p_n(t,x,y)|\leq K.
\]
Then, for $|h|$ small, we have
\begin{align*}
w(t,y)\abs{\frac{p(t,x,y+h)-p(t,x,y)}{h}}
&= \limsup _{n\to\infty }w(t,y)\abs{\frac{p_n(t,x,y+h)-p_n(t,x,y)}{h}}\\
&\leq \limsup _{n\to\infty }w(t,y)\int_0^1\left|\nabla p_n(t,x,y+sh)\right|ds\\
&\leq K\int_0^1\frac{w(t,y)}{w(t,y+sh)}ds.
\end{align*}
If we now let $|h|\to 0$, we obtain the desired inequality.
\end{proof}

As a simple consequence one obtains the following Sobolev regularity for $p$.
\begin{cor}\label{c.sobolev}
Assume in addition to Hypotheses \ref{h.3} and \ref{h.4}, that $\int_{Q(a,b)}w(t,x)^{-r}\, dt dx<\infty$ for some $r\in (1,\infty)$. Then
$p\in W^{0,1}_r(Q(a,b))$.
\end{cor}

\subsection{Polynomially growing coefficients}\label{s.appl}

\smallskip\noindent
Here we apply the results of the previous sections to the case of operators with polynomial diffusion coefficients, drift and potential terms.

Consider $Q(x)=(1+|x|^m_*)I$, $F(x)=-|x|^{p-1}x$ and $V(x)=|x|^s$ with $p>(m-1)\vee 1$, $s>|m-2|$ and $m>0$. To apply Theorem \ref{Thm: weighted-gradient estimate} we set
\[
w(t,x)=e^{\varepsilon t^\alpha |x|_*^\beta} \hbox{\ and } W_j(t,x)=e^{\varepsilon_j t^\alpha |x|_*^\beta},
\]
for $(t,y)\in (0,1)\times \R^d$, where $j=1,2$, $\beta=\frac{s-m+2}{2},\,0<2k\varepsilon<\varepsilon_1<\varepsilon_2<\frac{1}{\beta}$ and $\alpha>\frac{\beta}{\beta+m-2}$.

\begin{theo}\label{Thm: weighted-gradient estimate in polinomial case}
Let $p$ be the integral kernel associated with the operator $A$ with $Q(x)=(1+|x|^m_*)I$, $F(x)=-|x|^{p-1}x$ and $V(x)=|x|^s$, where $p>(m-1)\vee 1$, $s>|m-2|$ and $m> 0$. Then
\[
|\nabla p(t,x,y)|\leq C(1-\log t) t^{\frac{3}{2}-\frac{3\alpha (m\vee p\vee \frac{s}{2}) k+\alpha}{2\beta}} e^{-\varepsilon t^\alpha |x|_*^{\beta}}
\]
for $k>d+2$ and any $t\in (0,1),\,x,y\in \R^d$.
\end{theo}

\begin{proof}
\textit{Step 1.} We show that $W_1$ and $W_2$ are time dependent Lyapunov functions for $L=\partial_t+A$ with respect to the function
\[Z(x)=e^{\varepsilon_2 |x|_*^\beta}.\]
For that, we apply \cite[Proposition 3.3]{CKPR}.
Let $|x|\geq 1$ and set $G_j=\sum_{i=1}^d D_i q_{ij} = m |x|^{m-2}x_j.$
Since $s>|m-2|$, we have $\beta>(2-m)\vee 0$. 
Moreover,
\begin{align*}
|x|^{1-\beta-m} \left((G+F)\cdot \frac{x}{|x|}-\frac{V}{\varepsilon_j \beta |x|^{\beta -1}}\right)
&=  |x|^{1-\beta-m} \left(m |x|^{m-1}-|x|^p - \frac{|x|^s}{\varepsilon_j \beta |x|^{\beta -1}}\right)\\
&\leq m|x|^{-\beta} - \frac{1}{\varepsilon_j \beta}.
\end{align*}
If $|x|$ is large enough, for example $|x|\ge K$ with
$$K>\left( \frac{m}{\frac{1}{\varepsilon_j \beta}-1}\right)^\frac{1}{\beta}, $$
we get
\begin{align*}
|x|^{1-\beta-m} \left((G+F)\cdot \frac{x}{|x|}-\frac{V}{\varepsilon_j \beta |x|^{\beta -1}}\right)
\leq m|x|^{-\beta} - \frac{1}{\varepsilon_j \beta}
\leq m K^{-\beta}- \frac{1}{\varepsilon_j \beta}
<   -1,
\end{align*}
where we have used that $\varepsilon_j <\frac{1}{\beta}$.
In addition, we have
$$\lim_{\abs{x}\to\infty} V(x) \abs{x}^{2-2\beta-m}=\lim_{\abs{x}\to\infty} \abs{x}^{2-2\beta-m+s}=1.$$ 
Hence, $\lim_{\abs{x}\to\infty} V(x) \abs{x}^{2-2\beta-m}>c$ for any $c<1$. Consequently, by \cite[Proposition 3.3]{CKPR} we obtain that $W_1$ and $W_2$ are time dependent Lyapunov functions for $L=\partial_t+A$.
Similar computations show that the functions $Z(x)$ and $Z_0(x)=\exp (\varepsilon_2 |x|_*^{p+1-m})$ satisfy, respectively, Hypothesis \ref{h.1}(b) and (c).

\textit{Step 2.} We now show that $A$ satisfies Hypotheses \ref{h.3} and \ref{h.4}.
Fix $T=1,\,x\in \R^d$, $0<a_0<a<b<b_0<T$ and  $k>2(d+2)$.  
Let $(t,y)\in [a_0,b_0]\times\R^d$. Clearly, Hypothesis \ref{h.3}\textit{(a)}-\textit{(b)} and Hypothesis \ref{h.4}\textit{(b)} are satisfied. We assume that $|y|\geq 1$; otherwise, in a neighborhood of the origin, all the quantities we are going to estimate are obviously bounded.

First, since $2\varepsilon<\varepsilon_1$, we infer that
\[
w\le c_1w^{\frac{k-2}{k}}W_1^{\frac{1}{k}}
\]
with $c_1=1$. Second, we have
\begin{align}\label{eq1: Thm: weighted-gradient estimate in polinomial case}
\frac{|Q(y)\nabla w(t,y)|}{W_1(t,y)^{\frac{1}{2k}}}
&= \varepsilon \beta t^\alpha |y|^{\beta-1}(1+|y|^m) e^{-\frac{1}{2k}(\varepsilon_1-2k\varepsilon)t^\alpha |y|^\beta}\notag\\
&\leq 2\varepsilon \beta t^\alpha |y|^{\beta+m-1} e^{-\frac{1}{2k}(\varepsilon_1-2k\varepsilon)t^\alpha |y|^\beta}.
\end{align}
We make use of the following remark: since the function $t\mapsto t^p e^{-t}$ on $(0,\infty)$ attains its maximum at the point $t=p$, then for $\tau, \gamma, z>0$ we have
\begin{equation}\label{eq2: Thm: weighted-gradient estimate in polinomial case}
z^\gamma e^{-\tau z^\beta}
=\tau^{-\frac{\gamma}{\beta}}(\tau z^\beta)^\frac{\gamma}{\beta} e^{-\tau z^\beta}
\leq \tau^{-\frac{\gamma}{\beta}} \left( \frac{\gamma}{\beta}\right)^\frac{\gamma}{\beta} e^{-\frac{\gamma}{\beta}}
=: C(\gamma,\beta) \tau^{-\frac{\gamma}{\beta}}.
\end{equation}
Applying \eqref{eq2: Thm: weighted-gradient estimate in polinomial case} to the inequality \eqref{eq1: Thm: weighted-gradient estimate in polinomial case} with $z=|y|$, $\tau=\frac{1}{2k}(\varepsilon_1-2k\varepsilon)t^\alpha$, $\beta=\beta$ and $\gamma=\beta+m-1>0$ yields
\begin{align*}
\frac{|Q(y)\nabla w(t,y)|}{W_1(t,y)^{\frac{1}{2k}}}
\leq 2C(\beta+m-1, \beta)\varepsilon \beta t^\alpha \left[\frac{1}{2k}(\varepsilon_1-2k\varepsilon)t^\alpha\right]^{-\frac{\beta+m-1}{\beta}} 
\leq \overline{c} t^{-\frac{\alpha (m-1)^+}{\beta}}
\leq \overline{c} a_0^{-\frac{\alpha (m-1)^+}{\beta}}.
\end{align*}
Thus, we choose $c_2=\overline{c} a_0^{-\frac{\alpha (m-1)^+}{\beta}}$, where $\overline{c}$ is a universal constant.
In a similar way,
\begin{align*}
\frac{|Q(y)D^2 w(t,y)|}{W_1(t,y)^{\frac{1}{k}}}
&=\frac{(1+|y|^m)|D^2 w(t,y)|}{W_1(t,y)^{\frac{1}{k}}}\\
&\leq 2\sqrt{3}\varepsilon \beta t^\alpha \left[\left((\beta-2)^++\sqrt{d}\right) |y|^{\beta+m-2}+\varepsilon\beta t^{\alpha} |y|^{2\beta +m-2} \right] e^{-\frac{1}{k}(\varepsilon_1-k\varepsilon)t^\alpha |y|^\beta}. 
\end{align*}
Applying \eqref{eq2: Thm: weighted-gradient estimate in polinomial case} to each term, we get
\begin{align*}
\frac{|Q(y)D^2 w(t,y)|}{W_1(t,y)^{\frac{1}{k}}}
\leq  &C(\beta, m) \varepsilon \beta t^\alpha \left \lbrace  \left((\beta-2)^++\sqrt{d}\right)  \left[\frac{1}{k}(\varepsilon_1-k\varepsilon)t^\alpha\right]^{-\frac{\beta+m-2}{\beta}}\right. \\
&\left. +\varepsilon\beta t^{\alpha}\left[\frac{1}{k}(\varepsilon_1-k\varepsilon)t^\alpha\right]^{-\frac{2\beta+m-2}{\beta}}\right\rbrace
 \leq  \overline{c} t^{-\frac{\alpha(m-2)}{\beta}}\leq \overline{c} a_0^{-\frac{\alpha (m-2)^+}{\beta}}.
\end{align*}
Therefore, we pick $c_3=\overline{c} a_0^{-\frac{\alpha (m-2)^+}{\beta}}$. Furthermore, if we consider $t_0\in (0,t)$, we have
\begin{align*}
\frac{|Q(y)| |\nabla W_1(t_0,y)|}{ W_1(t_0,y) w(t,y)^{-1/k}W_1(t,y)^{1/2k}}
&=\sqrt{d} \beta\varepsilon_1 t_0^\alpha (1+|y|^m) |y|^{\beta-1}e^{-\frac{1}{2k}(\varepsilon_1-2\varepsilon)t^\alpha |y|^\beta}\\
&\leq  \overline{c} t^{-\alpha\frac{m-1}{\beta}}
\leq  \overline{c} a_0^{-\alpha\frac{(m-1)^+}{\beta}}
=: c_{12},
\end{align*}
where we used \eqref{eq2: Thm: weighted-gradient estimate in polinomial case}. We can proceed in the same way to check the remaining inequalities.
To sum up, the constants $c_1, \dots, c_{12}$ are the following:
\begin{align*}
&c_1=1, 
&&c_2=c_7=c_{12}=\overline{c} a_0^{-\frac{\alpha (m-1)^+}{\beta}},
&&c_3= \overline{c} a_0^{-\frac{\alpha (m-2)^+}{\beta}},\\
&c_4=c_{11}=\overline{c} a_0^{-1},
&&c_5=\overline{c}a_0^{-\frac{\alpha s}{2\beta}},
&&c_6=\overline{c}a_0^{-\frac{\alpha p}{\beta}},\\
&c_8=\overline{c} a_0^{-\frac{\alpha (p-1)}{\beta}},
&&c_9=\overline{c} a_0^{-\frac{\alpha (s-1)^+}{\beta}},
&&c_{10}=\overline{c}.
\end{align*}
\textit{Step 3.} We are now ready to apply Theorem \ref{Thm: weighted-gradient estimate}. 
To that end, we choose $a_0=\frac{t}{2}$, $a=t$, $b=\frac{3}{2}t$ and $b_0=2t$. If we now set $\lambda= m\vee p\vee \frac{s}{2}$, since $\alpha>\frac{\beta}{\beta+m-2}$, $s>|m-2|$ and $\beta=\frac{s-m+2}{2}$, we have
\[
\frac{\alpha\lambda}{\beta}>\frac{s}{2(\beta+m-2)}= \frac{s}{s+m-2}>\frac{1}{2}.
\]
Hence we can estimate the constants $A_1$ and $A_3$ in \eqref{def: constants A_i} as follows
\begin{align}\label{eq3: Thm: weighted-gradient estimate in polinomial case}
&A_1=c_1^{\frac{k}{2}}=1,\notag\\
&A_3
=c_6^k + c_2^\frac{k}{2} c_6^\frac{k}{2}+ c_5^k
=\overline{c} \left( t^{-\frac{\alpha pk}{\beta}} + t^{-\frac{\alpha((m-1)^++ p)k}{2\beta}}+t^{-\frac{\alpha sk}{2\beta}}\right)
\leq \overline{c} t^{-\frac{\alpha \lambda k}{\beta}}.
\end{align}
Similarly, if we consider the remaining constants in the right hand side of \eqref{RHS weighted-gradient estimate for p} we obtain that
\begin{align}\label{eq4: Thm: weighted-gradient estimate in polinomial case}
&\tilde{A}_2 \leq \overline{c} \left( t^{-\frac{\alpha \lambda k}{\beta}}+ t^{-\frac{k}{2}}\right),
&&B_1\leq\overline{c} t^{-\frac{\alpha \lambda }{\beta}},
&&B_2\leq \overline{c} t^{-\frac{\alpha \lambda}{\beta}-1},\notag\\
&B_3\leq \overline{c} t^{-\frac{3\alpha \lambda}{\beta}},
&&\tilde{B}_4\leq \overline{c}t^{-\frac{2\alpha \lambda}{\beta}},
&&B_5\leq \overline{c}t^{-\frac{2\alpha \lambda}{\beta}},\notag\\
&\tilde{B}_6\leq \overline{c}t^{-\frac{\alpha \lambda k}{\beta}},
&&B_7\leq \overline{c}t^{-\frac{\alpha \lambda k}{\beta}},
&&B_8\leq \overline{c}t^{-\frac{\alpha \lambda }{\beta}}.
\end{align}
Moreover, by \cite[Proposition 3.3]{CKPR}, there are two constants $H_1$ and $H_2$ not depending on $a_0$ and $b_0$ such that $\xi_{W_j}(t,x)\leq H_j$ for all $(s,x)\in [0,1]\times\R^d$, so for $j=1,2$ we have
\begin{equation}\label{eq5: Thm: weighted-gradient estimate in polinomial case}
\Xi_j(a_0,b_0)=\int_{a_0}^{b_0}\xi_{W_j}(t,x)\,dt
\leq H_j(b_0-a_0)= \frac{3t}{2} H_j.
\end{equation}
Furthermore, by Corollary \ref{Rmk: constants in h.2 under h.3}, we obtain
\[
p(t,x,y)\leq C t^{1-\frac{\alpha \lambda k}{\beta}} e^{-\varepsilon t^\alpha |y|_*^\beta}.
\] 
Then,
\begin{align*}
p(t,x,y)\log p(t,x,y)
&\leq C t^{1-\frac{\alpha \lambda k}{\beta}}  \left[ \log C + \left(1-\frac{\alpha \lambda k}{\beta}\right)\log t -\varepsilon t^\alpha |y|_*^\beta\right]e^{-\varepsilon t^\alpha |y|_*^\beta}\\
& \leq C  (1-\log t)t^{1-\frac{\alpha \lambda k}{\beta}} e^{-\varepsilon t^\alpha |y|_*^\beta} .
\end{align*}
Considering that $a=t$ and $b=\frac{3}{2}t$, it leads to
\begin{align}\label{eq6: Thm: weighted-gradient estimate in polinomial case}
\int_{\R^d} [p(t,x,y)\log p(t,x,y)]_{t=a}^{t=b}\,dy
&\leq C  (1-\log t) t^{1-\frac{\alpha \lambda k+\alpha}{\beta}}  \int_{\R^d} e^{-\varepsilon  |z|_*^\beta} dz \notag\\
&\leq C  (1-\log t) t^{1-\frac{\alpha \lambda k+\alpha}{\beta}} ,
\end{align}
where in the integrals we performed the change of variables $z=a^\frac{\alpha}{\beta}y$ and $z=b^\frac{\alpha}{\beta}y$.
We also get
\begin{align}\label{eq7: Thm: weighted-gradient estimate in polinomial case}
\int_{Q(a,b)}p(t,x,y)\log^2 p(t,x,y)\, dt\, dy
\leq C (1-\log t)^2 t^{2-\frac{\alpha \lambda k+\alpha}{\beta}}.
\end{align}
Putting \eqref{eq3: Thm: weighted-gradient estimate in polinomial case}-\eqref{eq7: Thm: weighted-gradient estimate in polinomial case} in \eqref{RHS weighted-gradient estimate for p} yields
\[
K\leq  C(1-\log t) t^{\frac{3}{2}-\frac{3\alpha \lambda k+\alpha}{2\beta}}. \qedhere
\]
\end{proof}

\end{document}